\RequirePackage{fix-cm}
\documentclass[smallextended]{svjour3}       
\smartqed  
\usepackage{graphicx}
\usepackage{amssymb}
\usepackage{amsmath}
\usepackage{mathdots}
\usepackage{bm}
\usepackage{mathtools}
\usepackage{type1cm}
\usepackage{latexsym,enumitem}
\usepackage{dsfont}
\usepackage{bbm}
\usepackage{soul}
\usepackage{epic,tikz}
\usepackage{natbib}
\newcommand{ \cG}{\mathcal{G}}

\newcommand{ \cP}{\mathcal{P}}

\newcommand{ \cZ}{\mathcal{Z}}

\newcommand{ \fA}{\mathfrak{A}}

\newcommand{\Erw}{\mathbb{E}}

\newcommand{\N}{\mathbb{N}}
  
\newcommand{\Prob}{\mathbb{P}} 

\newcommand{\R}{\mathbb{R}}

\newcommand{ \sfA}{\mathsf{A}}

\newcommand{ \sfB}{\mathsf{B}}

\newcommand{ \sfv}{\mathsf{v}}  

\newcommand{ \sfx}{\mathsf{x}}
\newcommand{ \sfy}{\mathsf{y}}
\newcommand{ \sfm}{\mathsf{m}}

\newcommand{ \bfe}{\mathbf{e}}

\newcommand{ \eps}{\varepsilon}
\newcommand{ \vph}{\varphi}
\newcommand{ \vth}{\vartheta}
\newcommand{\eqdist}{\stackrel{d}{=}}

\newcommand{\wh}{\widehat}

\newcommand{\ovl}{\overline}

\newcommand{ \1}{\mathbf{1}}

\newcommand{\BUTS}{\textsf{BUTS}}
\newcommand{\IBCOS}{\textsf{IBCOS}}

\allowdisplaybreaks[4] 

\journalname{J. Math. Biology}
\begin{document}

\title{The extinction problem for a distylous plant population with sporophytic self-incompatibility\thanks{This project has received funding from the European Research Council (ERC) under the European Union's Horizon 2020 research and innovation programme under the Grant Agreement No 759702 and from the Deutsche Forschungsgemeinschaft (SFB 878).}
}

\titlerunning{The extinction problem for a distylous plant population}        

\author{Gerold Alsmeyer   \and
        Kilian Raschel 
}


\institute{G. Alsmeyer \at
              Institut f\"ur Mathematische Stochastik, Fachbereich Mathematik und Informatik\\
              Universit\"at M\"unster, Orl\'eans-Ring 10\\
              D-48149 M\"unster, Germany\\
              \email{gerolda@uni-muenster.de}           
           \and
           K. Raschel \at
           CNRS, Institut Denis Poisson\\
           Universit\'e de Tours, Parc de Grandmont\\
           F-37200 Tours, France\\
            \email{raschel@math.cnrs.fr}
}

\date{Received: date / Accepted: date}

\maketitle

\begin{abstract}
In this paper, the extinction problem for a class of distylous plant populations is considered
within the framework of certain nonhomogeneous nearest-neighbor random walks in the positive quadrant. For the latter, extinction means absorption at one of the axes. Despite connections with some classical probabilistic models (standard two-type Galton-Watson process, two-urn model), exact formulae for the probabilities of absorption seem to be difficult to come by and one must therefore resort to good approximations. In order to meet this task, we develop potential-theoretic tools and provide various sub- and super-harmonic functions which, for large initial populations, provide bounds which in particular improve those that have appeared earlier in the literature. 
\keywords{extinction probability \and Markov jump process \and random walk \and branching process \and random environment \and sub- and superharmonic function \and potential theory}
\subclass{60J05 \and 60H25 \and 60K05}
\end{abstract}

\section{Introduction}\label{sec:intro}

In distylous flowering plant populations, where each plant belongs to one of two classes,
sporophytic self-incompatibility means that every plant produces pollen that can only fertilize the stigmata of plants from the opposite class but not from its own class.
A general model for such populations was developed by \cite{BilliardTran:12}, which allowed them to study different relationships between mate availability and fertilization success and to compare the dynamics of distylous species and self-fertile species. An important problem in this context is to find the probability of extinction for one of the phenotypes or at least good approximations thereof. In \cite{LafRaschTran:13}, this is done under the following specific assumptions which are also the basis of the present article:
\begin{itemize}
     \item Each plant in the population is diploid and its phenotype characterized by the two alleles it carries at a particular locus.
     \item There are two allelic types, say $\sfA$ and $\sfB$, the last one being dominant. Hence, the possible genotypes of the plants are $\sfA\sfA$, $\sfA\sfB$ and $\sfB\sfB$, the resulting phenotypes, i.e., types of proteins carried by their pollen and stigmates, being $\sfA,\,\sfB$ and $\sfB$, respectively.
     \item Due to self-incompatibility, only pollen and stigmates with different proteins can give viable seeds, i.e., pollen of a plant of phenotype $\sfA$ can only fertilize stigmates of a plant of phenotype $\sfB$ and vice versa.
\end{itemize}
By the last assumption, seeds of type $\sfB\sfB$ cannot be created. One may therefore consider, without loss of generality, populations made of individuals of genotypes $\sfA\sfA$ and $\sfA\sfB$ only. Each seed is then necessarily also of one of these two genotypes, with probability $1/2$ each. It is assumed that ovules are produced in continuous time at rate $r>0$ and that there is no pollen limitation, that is, each ovule is fertilized to give a seed provided there exists compatible pollen in the population. The lifetime of each individual is supposed to follow an exponential distribution with mean $1/d$, where $d>0$. Denoting by $N_{t}^{\sfA}$ and $N_{t}^{\sfB}$ the number of individuals of genotype $\sfA\sfB$ (phenotype $\sfA$) and $\sfB\sfB$ (phenotype $\sfB$) at time $t\in\R_{+}$, the process $(N_{t}^{\sfA},N_{t}^{\sfB})_{t\ge 0}$ forms a Markov jump process on the quarter plane $\N_{0}^{2}:=\{0,1,2,\ldots\}^2$ with transition rates on the interior $\N^{2}:=\{1,2,3,\ldots\}^2$ displayed in the left panel of Fig.~\ref{fig:transition rates}. The associated jump chain $(X_{n},Y_{n})_{n\ge 0}$, also called embedded Markov chain and obtained by evaluation of $(N_{t}^{\sfA},N_{t}^{\sfB})_{t\ge 0}$ at its jump epochs, then has transition probabilities (displayed in the right panel of Fig.~\ref{fig:transition rates})
\begin{align}
\begin{split}\label{eq:transition probabilities}
&\hspace{1.4cm}p_{(x,y),(x+1,y)}\ =\ p_{(x,y),(x,y+1)}\ =\ \frac{\ovl{\lambda}}{2},\\
&p_{(x,y),(x-1,y)}\ =\ \lambda\cdot\frac{x}{x+y},\quad p_{(x,y),(x,y-1)}\ =\ \lambda\cdot\frac{y}{x+y}
\end{split}
\end{align}
for $x,y\in\N^{2}$, where
$$ \lambda\ :=\ \frac{d}{d+r}\quad\text{and}\quad\ovl{\lambda}\ :=\ 1-\lambda. $$
Notice that it may also be viewed as a spatially nonhomogeneous nearest-neighbor random walk in the quarter plane. Self-incompatibility implies that reproduction becomes impossible and thus extinction occurs once one of the phenotypes disappears, i.e., $N_{t}^{\sfA}=0$ or $N_{t}^{\sfB}=0$ for some $t$. Consequently, both introduced processes are absorbed when hitting one of the axes.

\unitlength=1.3cm
\begin{figure}[ht!]
\centering 
\hspace{-1.4cm} 
 \begin{tabular}{cc}

    \begin{picture}(3.3,4.5)
    \thicklines
    \put(1,1){{\vector(1,0){3.5}}}
    \put(1,1){\vector(0,1){3.5}}
    \put(4.2,0.68){$N_{t}^{\sfA}$}
    \put(0.55,4.3){$N_{t}^{\sfB}$}
    \put(0.65,2.9){$y$}
    \put(2.9,0.68){$x$}
    \thinlines
    \put(3,3){\vector(1,0){1}}
    \put(3,3){\vector(-1,0){1}}
    \put(3,3){\vector(0,1){1}}
    \put(3,3){\vector(0,-1){1}}
    \put(3.95,3.1){$\displaystyle\frac{r(x+y)}{2}$}
    \put(1.68,2.7){$dx$}
    \put(3.1,1.77){$dy$}    
    \put(2.7,4.2){$\displaystyle\frac{r(x+y)}{2}$}
   \linethickness{0.1mm}
    \put(1,2){\dottedline{0.1}(0,0)(3.5,0)}
    \put(1,3){\dottedline{0.1}(0,0)(3.5,0)}
    \put(1,4){\dottedline{0.1}(0,0)(3.5,0)}
    \put(2,1){\dottedline{0.1}(0,0)(0,3.5)}
    \put(3,1){\dottedline{0.1}(0,0)(0,3.5)}
    \put(4,1){\dottedline{0.1}(0,0)(0,3.5)}
    \end{picture}\hspace{2cm}
&   \begin{picture}(4,4)
    \thicklines
    \put(1,1){{\vector(1,0){3.5}}}
    \put(1,1){\vector(0,1){3.5}}
    \put(4.2,0.68){$X_{n}$}
    \put(0.6,4.3){$Y_{n}$}
    \put(0.65,2.9){$y$}
    \put(2.9,0.68){$x$}
    \thinlines
    \put(3,3){\vector(1,0){1}}
    \put(3,3){\vector(-1,0){1}}
    \put(3,3){\vector(0,1){1}}
    \put(3,3){\vector(0,-1){1}}
    \put(4.1,3.1){$\displaystyle\frac{\overline{\lambda}}{2}$}
    \put(1.3,2.7){$\displaystyle\lambda\,\frac{x}{x+y}$}
    \put(3.1,1.77){$\displaystyle\lambda\,\frac{y}{x+y}$}    
    \put(2.67,4.1){$\displaystyle\frac{\overline{\lambda}}{2}$}
   \linethickness{0.1mm}
    \put(1,2){\dottedline{0.1}(0,0)(3.5,0)}
    \put(1,3){\dottedline{0.1}(0,0)(3.5,0)}
    \put(1,4){\dottedline{0.1}(0,0)(3.5,0)}
    \put(2,1){\dottedline{0.1}(0,0)(0,3.5)}
    \put(3,1){\dottedline{0.1}(0,0)(0,3.5)}
    \put(4,1){\dottedline{0.1}(0,0)(0,3.5)}
    \end{picture}
    \end{tabular}
\vspace{-0.8cm}
\caption{Transition rates at $(x,y)$ for the population-size Markov jump process $(N_{t}^{\sfA},N_{t}^{\sfB})_{t\ge 0}$ (left) and the transition probabilities for its associated embedded discrete-time Markov chain $(X_{n},Y_{n})_{n\ge 0}$ (right).}
\label{fig:transition rates}
\end{figure}
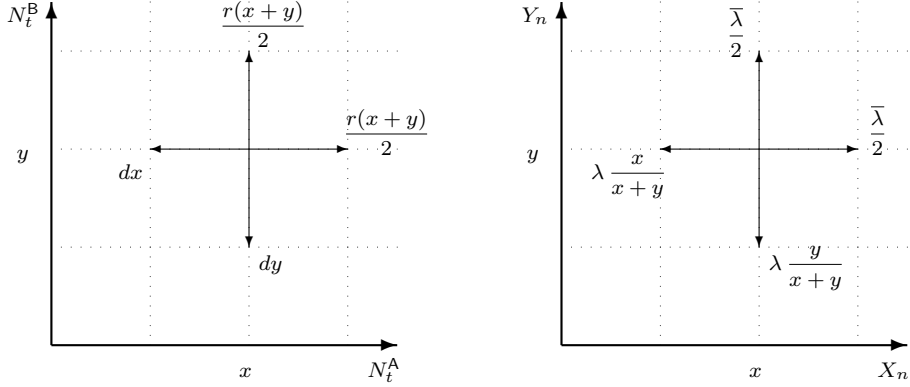

\vspace{.1cm}
Defining the extinction probabilities
\begin{equation}\label{eq:def q(x,y)}
q_{x,y}\ :=\ \Prob_{x,y}(\tau<\infty),
\end{equation}
where $\Prob_{x,y}:=\Prob(\cdot|X_{0}=x,Y_{0}=y)$ and
\begin{equation*}
\tau\ :=\ \inf\{n\ge 0:X_{n}=0\text{ or }Y_{n}=0\},
\end{equation*}
it is of natural biological interest to compute these probabilities. As it seems impossible to find explicit formulae, which might appear surprising at first glance, there is need for good approximations, lower and upper bounds, and asymptotic estimates as $x$ and/or $y$ tend to infinity. This is the main topic of the present article. A first step in this direction was made in Prop.~9 by \cite{BilliardTran:12}, who used a coupling argument to show that
\begin{gather}\label{eq:basic inequality q(x,y)}
\mu^{x+y}\ \le\ q_{x,y}\ \le\ \mu^{x}+\mu^{y}-\mu^{x+y}\ =\ 1-(1-\mu^{x})(1-\mu^{y})
\intertext{for all $x,y\in\N_{0}$ if}
\mu\,:=\,\frac{d}{r}\,<\,1\hspace{1cm}\text{(supercritical case)},\nonumber
\end{gather}
while $q_{x,y}=1$ for all $x,y\in\N_{0}$ if $\mu\ge 1$. In view of their result, we may focus on the supercritical case hereafter and therefore make the \emph{standing assumption} $\mu<1$, equivalently $\lambda=\frac{\mu}{1+\mu}<\frac{1}{2}$. 

\vspace{.1cm}
The lower bound in \eqref{eq:basic inequality q(x,y)} equals the extinction probability in a modification of our model in which evolution of the two subpopulations is the same up to $\tau$, but differs afterwards by having the one still alive at $\tau$ (the surviving type) to behave like an ordinary cell-splitting process. Indeed, $\mu^{x+y}$ then equals the probability that this type dies out eventually as well. To see this, we note that the total generation-size in the modified model forms an ordinary binary splitting Galton-Watson branching process $(W_{n})_{n\ge 0}$, say, with offspring distribution $(r_{j})_{j\ge 0}$, defined by
\begin{equation*}
     r_{0}=\lambda,\quad r_{2}=\ovl{\lambda},\quad\text{and}\quad r_{j}=0\text{ otherwise}.
\end{equation*}
Since, given $W_{0}=x+y$, its extinction probability equals $(r_{0}/r_{2})^{x+y}=\mu^{x+y}$, the assertion follows. Details of the modified model are provided in Subsect.~\ref{subsec:2-type splitting}.

\vspace{.1cm}
For an interpretation of the upper bound in \eqref{eq:basic inequality q(x,y)}, we refer to the next section where the given model is discussed in connection with related modifications as well as equivalent variants. The latter refers to models where a sequence $(X_{n},Y_{n})_{n\ge 0}$ of identical stochastic structure appears, although its interpretation may be different, see Subsec.~\ref{subsec:2-type splitting}--\ref{subsec:MCRE} as well as Tab.~\ref{tab:different_models} giving a tabular view of the models appearing this work.

\vspace{.1cm}
Apart from the rather crude inequality \eqref{eq:basic inequality q(x,y)}, the only further result, obtained in Prop.~2.3 and Rem.~2.4 by \cite{LafRaschTran:13}, asserts that for any fixed $y\in\N$ and $x\to\infty$,
\begin{equation}\label{eq:q_xy for fixed y}
q_{x,y}\ \simeq\ (2\mu)^{y}\frac{y!}{x^y}.
\end{equation}
On the other hand, the behavior of $q_{x,y}$ as $x$ and $y$ both get large (including the case of particular interest when $x=y$) appears to be completely open. As two relevant quantities, we mention here the lower and upper rates of exponential decay of $q_{x,x}$, viz.
\begin{equation}\label{eq:rate of decay of q_xx}
\kappa_{*}\ :=\ \liminf_{x\to\infty}q_{x,x}^{1/x}\quad\text{and}\quad\kappa^{*}\ :=\ \limsup_{x\to\infty}q_{x,x}^{1/x},
\end{equation}
which, by inequality \eqref{eq:basic inequality q(x,y)}, are positive and satisfy
\begin{equation}\label{eq:first bound kappa}
\mu^{2}\ \le\ \kappa_{*}\ \le\ \kappa^{*}\ \le\ \mu.
\end{equation}
Hence, for any $\kappa_{1}<\kappa_{*}$ and $\kappa_{2}>\kappa^{*}$ and all sufficiently large $x$, we have that $\kappa_{1}^{x}<q_{x,x}<\kappa_{2}^{x}$. Even the determination  of these rates $\kappa_{*}$ and $\kappa^{*}$, including the question whether or not they are equal, poses a very difficult problem. An improvement of \eqref{eq:basic inequality q(x,y)} and \eqref{eq:first bound kappa} will be stated as Thm.~\ref{thm:annealed inequality} and derived in Sec.~\ref{sec:proof theorem 1} along with an interpretation of the bounds. Our arguments are quite different from those in the afore-mentioned articles and based on a potential-theoretic analysis. One of the crucial objects in this analysis is the transition operator of $(X_{n},Y_{n})_{n\ge 0}$, denoted $P$ and formally defined as
$$ Pf(x,y)\ :=\ \Erw_{x,y}f(X_{1},Y_{1}) $$
for nonnegative functions $f$ on $\N_{0}^{2}$. In words, given initial population sizes $(X_{0},Y_{0})=(x,y)$, $Pf(x,y)$ provides the expected value of $f(X_1,Y_1)$ for an arbitrary  function $f$ as stated and $(X_1,Y_1)$ denoting the population sizes after one time step. Putting $x\wedge y:=\min\{x,y\}$ and $x\vee y:=\max\{x,y\}$, we have $Pf(x,y)=f(x,y)$ if $x\wedge y=0$ and
\begin{align}
\begin{split}\label{eq:transition operator P}
Pf(x,y)\ &= \lambda\left(\frac{x}{x+y}\,f(x-1,y)+\frac{y}{x+y}\,f(x,y-1)\right)\\
&\quad+\ \ovl{\lambda}\,\frac{f(x+1,y)+f(x,y+1)}{2}
\end{split}
\end{align}
for $x,y\in\N$. As one can readily verify, $q(x,y)=q_{x,y}$ is $P$-harmonic, i.e., $Pq=q$ or, in explicit form,
\begin{align}\label{eq:q(x,y) P-harmonic}
q_{x,y}\ =\ \lambda\left(\frac{x}{x+y}\,q_{x-1,y}+\frac{y}{x+y}\,q_{x,y-1}\right)\,+\,\ovl{\lambda}\,\frac{q_{x+1,y}+q_{x,y+1}}{2}
\end{align}
for all $x,y\in\N_{0}$. Hence $q$ forms a solution to the Dirichlet problem
\begin{align*}
Pf(x,y)=f(x,y)\quad\text{and}\quad f(x,0)=f(0,y)=1\quad\text{for all }x,y\in\N.
\end{align*}
There are infinitely many positive solutions, for instance $f\equiv 1$, among which $q$ constitutes the minimal one as has been shown in Prop.~2.1 by \cite{LafRaschTran:13} by a standard martingale argument. It is also shown there that in fact any specification of $f(x,1)$ for $x\in\N$ uniquely determines a solution. By the symmetries of the model, any solution $f$ must further satisfy $f(x,y)=f(y,x)$ for all $x,y\in\N_{0}$.

\vspace{.1cm}
In lack of an explicit formula for $q_{x,y}$ and towards finding upper and lower bounds thereof, a natural and quite standard potential-theoretic approach is to look for functions $f(x,y)$ which are sub- or superharmonic with respect to $P$, i.e., $Pf\le f$ or $Pf\ge f$, and satisfy the same boundary conditions as $q_{x,y}$, thus $f(x,y)=1$ if $x\wedge y=0$. Then, it is a well-known fact that $P^{n}f$, where $P^{n}$ denotes the $n$-fold iteration of $P$, decreases or increases to $q$. As a consequence, $P^{n}f(x,y)$, if computable, provides an upper or lower bound for $q_{x,y}$. This approach will be used to derive Thm.~\ref{thm:annealed inequality}.

\begin{table}[h!]\vspace{24mm}
\setlength{\unitlength}{8.5mm}
\begin{center}
\begin{tabular}{ll}
\begin{picture}(2,2)
    \thicklines
    \put(3,3){\vector(1,1){1}}
    \put(3,3){\vector(-1,-1){1}}
    \put(4.1,4.1){$\overline{\lambda}=1-\lambda$}
    \put(1.65,1.75){$\lambda$}
    \end{picture} 
&\begin{picture}(2,2)
    \thicklines
    \put(3,3){\vector(1,0){1}}
    \put(3,3){\vector(-1,0){1}}
    \put(3,3){\vector(0,1){1}}
    \put(3,3){\vector(0,-1){1}}
    \put(4.1,3.1){$\displaystyle\frac{\overline{\lambda}}{2}$}
    \put(1.0,2.7){$\displaystyle\lambda\,\frac{x}{x+y}$}
    \put(3.1,1.77){$\displaystyle\lambda\,\frac{y}{x+y}$}    
    \put(2.55,4){$\displaystyle\frac{\overline{\lambda}}{2}$}
    \end{picture} \vspace{-12mm}
 \\ \begin{minipage}{5cm}One-dimensional random walk on $\mathbb N_0$. It describes the total population size $X_n+Y_n$ in all discussed models and is used to get the lower bound $\mu^2$ for the exponential decay.
 \end{minipage}&
\begin{minipage}{5cm}Main model studied in this paper, see Fig.~\ref{fig:transition rates} in the Introduction.\vspace{10.0mm}
 \end{minipage} \vspace{9mm}\\
\begin{picture}(4,4)
    \thicklines
\put(3,3){\vector(1,0){1}}
    \put(3,3){\vector(-1,0){1}}
    \put(3,3){\vector(0,1){1}}
    \put(3,3){\vector(0,-1){1}}
    \put(4.1,3.2){$\ovl\lambda\displaystyle\frac{x}{x+y}$}
    \put(1.1,2.7){$\lambda\displaystyle\frac{x}{x+y}$}
    \put(3.1,1.77){$\lambda\displaystyle\frac{y}{x+y}$}    
    \put(2.2,4.4){$\ovl\lambda\displaystyle\frac{y}{x+y}$}
    \end{picture} 
 & \begin{picture}(4,4)
    \thicklines
        \put(3,3){\vector(1,0){1}}
    \put(3,3){\vector(-1,0){1}}
    \put(3,3){\vector(0,1){1}}
    \put(3,3){\vector(0,-1){1}}
    \put(4.1,3.1){$\displaystyle\frac{\overline{\lambda}}{2}$}
    \put(1.6,2.7){$\displaystyle\frac{\lambda}{2}$}
    \put(3.1,1.77){$\displaystyle\frac{\lambda}{2}$}    
    \put(2.6,4.1){$\displaystyle\frac{\overline{\lambda}}{2}$}
    \end{picture} \vspace{-10mm}
 \\ \begin{minipage}{5cm}Model \IBCOS, see Subsec.~\ref{subsec:BMHOTS}.\vspace{6.6mm}
 \end{minipage} & \begin{minipage}{5cm}Model \BUTS, see Subsec.~\ref{subsec:BMHOTS}. Used to provide the upper bound $\mu$ for the exponential decay rate.
 \end{minipage}\vspace{7mm}\\
 \begin{picture}(4,4)
    \thicklines
    \put(3,3){\vector(1,0){1}}
    \put(3,3){\vector(-1,0){1}}
    \put(3,3){\vector(0,1){1}}
    \put(3,3){\vector(0,-1){1}}
    \put(3.7,3.2){$\ovl{\lambda}\,\phi(x,y)$}
    \put(0.8,2.6){${\lambda}\,{\phi}(x,y)$}
    \put(2.6,1.6){${\lambda}\,\ovl{\phi}(x,y)$}    
    \put(2,4.2){$\ovl{\lambda}\,\ovl{\phi}(x,y)$}
      \end{picture} 
&\begin{picture}(4,4)
    \thicklines
    \put(3,3){\vector(1,0){1}}
    \put(3,3){\vector(-1,0){1}}
    \put(3,3){\vector(0,1){1}}
    \put(3,3){\vector(0,-1){1}}
    \put(3.7,3.2){$\ovl{\lambda}\,\phi_{+}(x,y)$}
    \put(0.8,2.6){${\lambda}\,{\phi}_{-}(x,y)$}
    \put(2.6,1.6){${\lambda}\,\ovl{\phi}_{-}(x,y)$}    
    \put(2,4.2){$\ovl{\lambda}\,\ovl{\phi}_{+}(x,y)$}
      \end{picture}   \vspace{-10mm}
  \\ \begin{minipage}{5cm}Branching with homogeneous type selection. Type selection probabilities $\phi(x,y)$ are the same for births and deaths. \IBCOS\ and \BUTS\  are special cases, see \eqref{eq:transition probabilities hom} and Subsec.~\ref{subsec:BMHOTS}.
 \end{minipage} & \begin{minipage}{5cm}The hybrid model with type selection probabilities $\phi_{\pm}(x,y)$ that may vary for births and deaths. Main model is special case, see \eqref{eq:transition probabilities hybrid} and Subsec.~\ref{subsec:BMHOTS}.\vspace{3.2mm}
 \end{minipage}\vspace{5mm}\\
   \begin{picture}(4,4)
    \thicklines  
     \put(3,3){\vector(1,0){1}}
    \put(3,3){\vector(-1,0){1}}
    \put(3,3){\vector(0,1){1}}
    \put(3,3){\vector(0,-1){1}}
    \put(3,3){\vector(1,1){1}}
    \put(3,3){\vector(-1,1){1}}
    \put(3,3){\vector(1,-1){1}}
    \put(3,3){\vector(-1,-1){1}}
    \put(4.15,4.2){$p_{1,1}$}
    \put(4.15,3){$p_{1,0}$}
    \put(1.1,3){$p_{-1,0}$}
     \put(1.05,1.8){$p_{-1,-1}$}
     \put(2.7,1.8){$p_{0,-1}$}
    \put(4.15,1.8){$p_{1,-1}$}
    \put(2.8,4.2){$p_{0,1}$}
    \put(1.4,4.2){$p_{-1,1}$}
    \end{picture} 
    &\begin{picture}(4,4)
    \thicklines
    \put(3,3){\vector(-1,0){1}}
    \put(3,3){\vector(1,1){1}}
    \put(3,3){\vector(0,-1){1}}
    \put(0.6,2.6){$\dfrac{x\,\mu}{\lambda+x\mu+y\nu}$}
    \put(3.1,1.77){$\dfrac{y\,\nu}{\lambda+x\mu+y\nu}$}    
    \put(3.7,3.5){$\dfrac{\lambda}{\lambda+x\mu+y\nu}$}
 \end{picture}\vspace{-10mm} \\
    \begin{minipage}{5cm}Generalization of the \BUTS\ model, see Fig.~\ref{fig:Foddy} and Subsec.~\ref{subsec:related}.\vspace{3.1mm}
 \end{minipage} & \begin{minipage}{5cm}M/M/$\infty$ parallel queuing model with simultaneous arrivals by \cite{Foddy:84}, see Fig.~\ref{fig:Foddy} and Subsec.~\ref{subsec:related}.
 \end{minipage}
\end{tabular}
\end{center}
\caption{A tabular view of the models appearing in this paper}
\label{tab:different_models}
\end{table}

\vspace{.1cm}
In order to motivate our second main result, Thm.~\ref{thm:q_x,y asymptotic}, we first point out that the present model appears to be a hybrid, in a sense to be explained in Subsec.~\ref{subsec:BMHOTS}, of two 2-type population models. Moreover, as stated in Prop.~\ref{prop:extinction IBCOS BUTS}, the upper bound $\mu^{x}+\mu^{y}-\mu^{x+y}$ in \eqref{eq:basic inequality q(x,y)} equals the \emph{exact} extinction probability for both of these models when initially given $x$ individuals of type $\sfA$ and $y$ individuals of type $\sfB$, whence $\kappa_{*}=\kappa^{*}=\mu$. This leads to the natural conjecture that the last statement remains true for the present model. However, it will be disproved by Thm.~\ref{thm:q_x,y asymptotic} which asserts that $\kappa^{*}<\mu$. Its proof will require an extended look at the population dynamics by viewing the inherent branching mechanism as a random environment. For details, we refer to Subsec.~\ref{subsec:MCRE}.

\vspace{.1cm}
We have organized the rest of this article as follows: A number of relevant models, related or equivalent to the given one, are presented and discussed in some detail in the next section, including a synoptic view provided by Tab.~\ref{tab:different_models} below. This allows us to put in perspective various aspects of the crucial sequence $(X_{n},Y_{n})_{n\ge 0}$ by defining it in different contexts as well as a particular random walk model within a larger class of similar ones. Our results and some interpretations are presented in Sec.~\ref{sec:results}, with proofs following in Sec.~\ref{sec:proof theorem 1}--\ref{sec:extinction IBCOS BUTS}.

\section{The model in alternative contexts and generalizations}\label{sec:2}

From a mathematical point of view, it is useful to see the random walk $(X_{n},Y_{n})_{n\ge 0}$ appearing in various contexts, three of which we shortly describe hereafter. The last of these will be particularly interesting because it offers an extended framework by introducing a random environment. This will allow us to look at the behavior of $(X_{n},Y_{n})_{n\ge 0}$ also on a quenched level and thus provide an additional leverage for the derivation of good bounds for the $q_{x,y}$, in fact a key tool for the proof of Thm.~\ref{thm:q_x,y asymptotic} in Sec.~\ref{sec:proof theorem 2}.

\subsection{Standard $2$-type binary splitting}\label{subsec:2-type splitting}
 
Consider a two-type binary splitting population model in continuous time, in which individuals act independently and any individual $\sfv$ has a type $\sigma_{\sfv}\in\{\sfA,\sfB\}$, a standard exponential lifetime and a random number $L_{\sfv}^{\sfx}$ of type-$\sfx$ offspring for $\sfx\in\{\sfA,\sfB\}$ which is produced at the end of her life and independent of the lifetime. Furthermore,
\begin{align*}
&\Prob((L_{\sfv}^{\sfA},L_{\sfv}^{\sfB})=(0,0)|\sigma_{\sfv}=\sfA)\,=\,\Prob((L_{\sfv}^{\sfA},L_{\sfv}^{\sfB})=(0,0)|\sigma_{\sfv}=\sfB)\,=\,\frac{d}{d+r}\,=\,\lambda,\\
&\Prob((L_{\sfv}^{\sfA},L_{\sfv}^{\sfB})=(1,1)|\sigma_{\sfv}=\sfA)\,=\,\Prob((L_{\sfv}^{\sfA},L_{\sfv}^{\sfB})=(2,0)|\sigma_{\sfv}=\sfA)\,=\,\frac{\ovl{\lambda}}{2},\\
&\Prob((L_{\sfv}^{\sfA},L_{\sfv}^{\sfB})=(1,1)|\sigma_{\sfv}=\sfB)\,=\,\Prob((L_{\sfv}^{\sfA},L_{\sfv}^{\sfB})=(0,2)|\sigma_{\sfv}=\sfB)\,=\,\frac{\ovl{\lambda}}{2}
\end{align*}
for parameters $r,d>0$ such that $d<r$. Denoting by $\cZ_{t}^{\sfx}$ the number of living individuals of type $\sfx$ at time $t$, the process $(\cZ_{t}^{\sfA},\cZ_{t}^{\sfB})_{t\ge 0}$ is a supercritical two-type Bellman-Harris branching process with mean reproduction matrix
\begin{align*}
\begin{pmatrix}
\sfm_{\sfA\sfA} &\sfm_{\sfA\sfB}\\ \sfm_{\sfB\sfA} &\sfm_{\sfB\sfB}
\end{pmatrix}
\ =\ 
\begin{pmatrix}
\displaystyle\frac{3r}{2(d+r)} &\displaystyle\frac{r}{2(d+r)}\\[3mm] \displaystyle\frac{r}{2(d+r)} &\displaystyle\frac{3r}{2(d+r)}
\end{pmatrix},
\end{align*}
where $\sfm_{\sfx\sfy}=\Erw(L_{\sfv}^{\sfy}|\sigma_{\sfv}=\sfx)$. Due to the symmetric reproduction mechanism, the total population size $\cZ_{t}=\cZ_{t}^{\sfA}+\cZ_{t}^{\sfB}$ at time $t$ forms a supercritical binary splitting Bellman-Harris process with offspring mean
$$ \sfm\ =\ \sfm_{\sfA\sfA}+\sfm_{\sfA\sfB}\ =\ \sfm_{\sfB\sfA}+\sfm_{\sfB\sfB}\ =\ 2\ovl{\lambda}\ >\ 1. $$
As a consequence, the pertinent total generation-size sequence $(Z_{n})_{n\ge 0}$ is an ordinary binary splitting Galton-Watson process with offspring mean $\sfm$, offspring distribution
\begin{equation*}
p_{0}\ =\ \lambda\ =\ 1-p_{2}
\end{equation*}
and extinction probability
\begin{equation*}
q\ =\ \frac{d}{r}\ =\ \mu
\end{equation*}
when starting with one ancestor $(Z_{0}=1)$. Clearly, $q$ is also the probability of extinction for $(\cZ_{t})_{t\ge 0}$.

\vspace{.1cm}
To make the connection with our original model, notice that $(\cZ_{t}^{\sfA},\cZ_{t}^{\sfB})_{t\ge 0}$ also constitutes a continuous-time birth-death process on $\N_{0}^{2}$ and has the same transition rates as $(N_{t}^{\sfA},N_{t}^{\sfB})_{t\ge 0}$ on $\N^{2}$. Consequently,
\begin{align*}
(N_{t}^{\sfA},N_{t}^{\sfB})_{t\ge 0}\,\eqdist\,(\cZ_{t\wedge T}^{\sfA},\cZ_{t\wedge T}^{\sfB})_{t\ge 0}\quad\text{and}\quad (X_{n},Y_{n})_{n\ge 0}\,\eqdist\,(Z_{n\wedge\nu}^{\sfA},Z_{n\wedge\nu}^{\sfB})_{n\ge 0},
\end{align*}
where $(Z_{n}^{\sfA},Z_{n}^{\sfB})_{n\ge 0}$ denotes the associated jump chain of $(\cZ_{t}^{\sfA},\cZ_{t}^{\sfB})_{t\ge 0}$,
\begin{align*}
T\,:=\,\inf\{t\ge 0:\cZ_{t}^{\sfA}\wedge\cZ_{t}^{\sfB}=0\},\quad\nu\,:=\,\inf\{n\ge 0:Z_{n}^{\sfA}\wedge Z_{n}^{\sfB}=0\}
\end{align*}
and $\eqdist$ denotes equality in law. We thus see that our extinction problem may be rephrased as an extinction problem for a particular $2$-type Galton-Watson process which, however, is nonstandard because for the latter process it means to find 
the probability for each of the types to disappear momentarily (by irreducibility, new individuals of that type may be produced afterwards as offspring from the other type). 

\subsection{A two-urn model}\label{subsec:urn model}

Another very simple way to obtain the sequence $(X_{n},Y_{n})_{n\ge 0}$ is by considering the following nonterminating two-step procedure of adding or removing a ball, one per round, from one of two urns, say $\sfA$ and $\sfB$. Initially, these urns contain $\wh{X}_{0}$ and $\wh{Y}_{0}$ balls, respectively. In the first step of each round, we toss a coin so as to determine whether a ball is removed or added, which happens with respective probabilities $0<\lambda<1/2$ and $\ovl{\lambda}$. Then, if a ball is to be removed, we just pick one at random not regarding the urn in which it lies. But if a ball is to be added, then we pick the designated urn at random. Let $\wh{X}_{n}$ and $\wh{Y}_{n}$ be the number of balls in $\sfA$ and $\sfB$ after $n$ rounds, respectively, and put
$$ \nu\,:=\,\inf\{n\ge 0:\wh{X}_{n}\wedge\wh{Y}_{n}=0\}. $$
Then $(X_{n},Y_{n}):=(\wh{X}_{n\wedge\nu},\wh{Y}_{n\wedge\nu})$, $n\in\N_{0}$, is indeed a random walk on $\N_{0}^{2}$ with transition probabilities and transition operator given by \eqref{eq:transition probabilities} and \eqref{eq:transition operator P}, respectively, and this time obtained from the nonhomogeneous random walk $(\wh{X}_{n},\wh{Y}_{n})_{n\ge 0}$ on $\N_{0}^{2}$ by killing the latter when it first hits one of the axes. Unfortunately, this nice alternative interpretation of our model does again not lead to any additional clue about how to solve our extinction problem.

\subsection{A Markov chain with iid random transition probabilities}\label{subsec:MCRE}

Rather than yet another alternative, our last model description should be seen as an extension of the one given in Subsec.~\ref{subsec:2-type splitting} by enlarging the perspective in some sense. Recall from there that $(Z_{n}^{\sfA},Z_{n}^{\sfB})_{n\ge 0}$ denotes the jump chain of the $2$-type Bellman-Harris process $(\cZ_{t}^{\sfA},\cZ_{t}^{\sfB})_{t\ge 0}$ and that $(Z_{n}^{\sfA}+Z_{n}^{\sfB})_{n\ge 0}$ has iid increments $e_{1},e_{2},\ldots$ taking values $-1$ and $+1$ with respective probabilities $\lambda$ and $\ovl{\lambda}$. Obviously, the value of $e_{n}$ determines whether the $n${th} jump epoch marks a birth ($+1$) or a death ($-1$) in the population. Let us adopt the perspective of $\bfe=(e_{n})_{n\ge 1}$ being a random environment for the Markov chain $(Z_{n}^{\sfA},Z_{n}^{\sfB})_{n\ge 0}$. Then, given $\bfe$, this sequence is still Markovian but temporally nonhomogeneous, its transition operator at time $n$ being $\cP_{e_{n}}$, where
\begin{align*}
\cP_{1}f(x,y)\ &:=\ \frac{f(x+1,y)+f(x,y+1)}{2}
\shortintertext{and}
\cP_{-1}f(x,y)\ &:=\ \frac{x}{x+y}\,f(x-1,y)+\frac{y}{x+y}\,f(x,y-1)
\end{align*}
for $x,y\in\N_{0}$. In other words,
\begin{align*}
&\Erw_{x,y}^{\bfe}(f(Z_{n}^{\sfA},Z_{n}^{\sfB})|Z_{n-1}^{\sfA},Z_{n-1}^{\sfB})\ =\ \cP_{e_{n}}f(Z_{n-1}^{\sfA},Z_{n-1}^{\sfB})\quad\text{a.s.}
\shortintertext{and}
&\Erw_{x,y}^{\bfe}f(Z_{n}^{\sfA},Z_{n}^{\sfB})\ =\ \cP_{e_{1}}\cP_{e_{2}}\cdots\cP_{e_{n}}f(x,y)\quad\text{a.s.}
\end{align*}
for all $n\in\N$ and $x,y\in\N_{0}$, where $\Erw_{x,y}^{\bfe}:=\Erw_{x,y}(\cdot|\bfe)$. Freezing $(Z_{n}^{\sfA},Z_{n}^{\sfB})_{n\ge 0}$ when it hits the axes leads to $(X_{n},Y_{n})_{n\ge 0}$, which in turn implies that
\begin{equation*}
\Erw_{x,y}^{\bfe}(f(X_{n},Y_{n})|X_{n-1},Y_{n-1})\ =\ P_{e_{n}}f(X_{n-1},Y_{n-1})\quad\text{a.s.}
\end{equation*}
and
\begin{equation*}
\Erw_{x,y}^{\bfe}f(X_{n},Y_{n})\ =\ P_{e_{1}}P_{e_{2}}\cdots P_{e_{n}}f(x,y)\quad\text{a.s.}
\end{equation*}
for all $n\in\N$ and $x,y\in\N_{0}$, where $P_{\pm 1}$ equals the modification of $\cP_{\pm 1}$ which is absorbing on the axes, i.e., $P_{\pm 1}(x,y)=f(x,y)$ if $x\wedge y=0$. So we see that, by introduction of $\bfe$, $(X_{n},Y_{n})_{n\ge 0}$ becomes a Markov chain with iid random transition probabilities, viz.
\begin{align}
\begin{split}\label{eq:random transition probabilities}
p_{(x,y),(x+1,y)}(e_{n})\ &=\ p_{(x,y),(x,y+1)}(e_{n})\ =\ \frac{1}{2}\,\1_{\{e_{n}=1\}},\\
p_{(x,y),(x-1,y)}(e_{n})\ &=\ \frac{x}{x+y}\,\1_{\{e_{n}=-1\}},\\
p_{(x,y),(x,y-1)}(e_{n})\ &=\ \frac{y}{x+y}\,\1_{\{e_{n}=-1\}}
\end{split}
\end{align}
for $x,y\in\N^{2}$ and $p_{(x,0),(x,0)}(e_{n})=p_{(0,y),(0,y)}(e_{n})=1$ for $x,y\in\N_{0}^{2}$. The associated quenched extinction probabilities are denoted by $q_{x,y}(\bfe)$, so
\begin{equation*}
q_{x,y}(\bfe)\ :=\ \Prob_{x,y}^{\,\bfe}(\tau<\infty)
\end{equation*}
for $x,y\in\N_{0}$. Plainly, $q_{x,0}(\bfe)=q_{0,y}(\bfe)=1$, and
$$ q_{x,y}\ =\ \Erw\,q_{x,y}(\bfe). $$

\subsection{2-type branching models with homogeneous type selection}\label{subsec:BMHOTS}

Our model may also be viewed as a particular instance of the following general 2-type branching model. As before, individuals can be of type $\sfA$ or $\sfB$ and
$X_{n},Y_{n}$ denote the type-$\sfA$ and type-$\sfB$ subpopulation sizes, respectively, after $n$ branching events (a splitting or a death). Given $\mu\in (0,1)$ and functions $\phi_{\pm}:\N^{2}\to (0,1)$ satisfying $\phi_{\pm}(x,y)=\phi_{\pm}(y,x)$, suppose that at each time $n$, given nonzero current subpopulation sizes $x$ and $y$, an individual dies with probability $\lambda$ and is born with probability $\ovl{\lambda}$. In the first case, this individual is of type $\sfA$ with probability $\phi_{-}(x,y)$ and of type $\sfB$ with probability $\ovl{\phi}_{-}(x,y)=1-\phi_{-}(x,y)$. In the second case, it is of type $\sfA$ with probability $\phi_{+}(x,y)$ and of type $\sfB$ with probability $\ovl{\phi}_{+}(x,y)$. It then follows that $(X_{n},Y_{n})_{n\ge 0}$ is a random walk on $\N_{0}^{2}$ with transition probabilities
\begin{align}
\begin{split}\label{eq:transition probabilities hybrid}
p_{(x,y),(x+1,y)}\ &=\ \ovl{\lambda}\,\phi_{+}(x,y),\quad p_{(x,y),(x,y+1)}\ =\ \ovl{\lambda}\,\ovl{\phi}_{+}(x,y),\\
p_{(x,y),(x-1,y)}\ &=\ \lambda\,\phi_{-}(x,y),\quad p_{(x,y),(x,y-1)}\ =\ \lambda\,\ovl{\phi}_{-}(x,y)
\end{split}
\end{align}
for $(x,y)\in\N^{2}$, see Tab.~\ref{tab:different_models}. Naturally, the walk is assumed to be absorbed at the axes. If $\phi_{-}=\phi_{+}=:\phi$, then we call this a \emph{2-type branching model with homogeneous type selection} and a \emph{hybrid model} otherwise. Our plant population model constitutes a special instance of a hybrid, with $\phi_{-}(x,y)=\frac{x}{x+y}$ and $\phi_{+}(x,y)=\frac{1}{2}$, of the following two models with homogeneous type selection, \IBCOS\ and \BUTS.

\unitlength=1.3cm
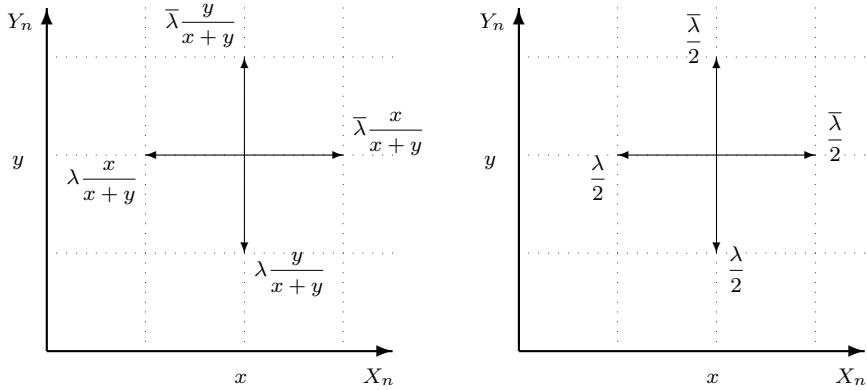
\begin{figure}
\centering 
\hspace{-1.6cm} 
 \begin{tabular}{cc}

    \begin{picture}(3.3,4.5)
    \thicklines
    \put(1,1){{\vector(1,0){3.5}}}
    \put(1,1){\vector(0,1){3.5}}
    \put(4.2,0.68){$X_{n}$}
    \put(0.6,4.3){$Y_{n}$}
    \put(0.65,2.9){$y$}
    \put(2.9,0.68){$x$}
    \thinlines
    \put(3,3){\vector(1,0){1}}
    \put(3,3){\vector(-1,0){1}}
    \put(3,3){\vector(0,1){1}}
    \put(3,3){\vector(0,-1){1}}
    \put(4.1,3.2){$\ovl\lambda\displaystyle\frac{x}{x+y}$}
    \put(1.2,2.7){$\lambda\displaystyle\frac{x}{x+y}$}
    \put(3.1,1.77){$\lambda\displaystyle\frac{y}{x+y}$}    
    \put(2.2,4.3){$\ovl\lambda\displaystyle\frac{y}{x+y}$}
   \linethickness{0.1mm}
    \put(1,2){\dottedline{0.1}(0,0)(3.5,0)}
    \put(1,3){\dottedline{0.1}(0,0)(3.5,0)}
    \put(1,4){\dottedline{0.1}(0,0)(3.5,0)}
    \put(2,1){\dottedline{0.1}(0,0)(0,3.5)}
    \put(3,1){\dottedline{0.1}(0,0)(0,3.5)}
    \put(4,1){\dottedline{0.1}(0,0)(0,3.5)}
    \end{picture}\hspace{1.5cm}
&   \begin{picture}(4,4)
    \thicklines
    \put(1,1){{\vector(1,0){3.5}}}
    \put(1,1){\vector(0,1){3.5}}
    \put(4.2,0.68){$X_{n}$}
    \put(0.6,4.3){$Y_{n}$}
    \put(0.65,2.9){$y$}
    \put(2.9,0.68){$x$}
    \thinlines
    \put(3,3){\vector(1,0){1}}
    \put(3,3){\vector(-1,0){1}}
    \put(3,3){\vector(0,1){1}}
    \put(3,3){\vector(0,-1){1}}
    \put(4.1,3.1){$\displaystyle\frac{\overline{\lambda}}{2}$}
    \put(1.7,2.7){$\displaystyle\frac{\lambda}{2}$}
    \put(3.1,1.77){$\displaystyle\frac{\lambda}{2}$}    
    \put(2.67,4.1){$\displaystyle\frac{\overline{\lambda}}{2}$}
   \linethickness{0.1mm}
    \put(1,2){\dottedline{0.1}(0,0)(3.5,0)}
    \put(1,3){\dottedline{0.1}(0,0)(3.5,0)}
    \put(1,4){\dottedline{0.1}(0,0)(3.5,0)}
    \put(2,1){\dottedline{0.1}(0,0)(0,3.5)}
    \put(3,1){\dottedline{0.1}(0,0)(0,3.5)}
    \put(4,1){\dottedline{0.1}(0,0)(0,3.5)}
    \end{picture}
    \end{tabular}
\vspace{-0.8cm}
\caption{Transition probabilities at $(x,y)$ for $(X_{n},Y_{n})_{n\ge 0}$ in the branching model with complete segregation (left panel) and the model with fully symmetric type selection (right panel). See also Tab.~\ref{tab:different_models}.}
\label{fig:homogeneous models}
\end{figure}

\vspace{.2cm}
\emph{Independent branching with complete segregation (\IBCOS)} occurs if the two subpopulations $\sfA$ and $\sfB$ evolve independently. Therefore, given current subpopulation sizes $x$ and $y$, a birth or death ``picks" an individual uniformly at random, that is, with probability $\phi(x,y)=\frac{x}{x+y}$ from $\sfA$ and probability $\ovl{\phi}(x,y)=\frac{y}{x+y}$ from $\sfB$. See Fig.~\ref{fig:homogeneous models} (left panel) and Tab.~\ref{tab:different_models}.

\vspace{.2cm}
In the \emph{branching with unbiased type selection (\BUTS)} model, it is the subpopulation (and thus the type) which is picked uniformly at random by the branching event (birth or death) and thus with probability $\phi(x,y)=\ovl{\phi}(x,y)=\frac{1}{2}$ each (before absorption). See Fig.~\ref{fig:homogeneous models} (right panel) and Tab.~\ref{tab:different_models}.

\vspace{.2cm}
Here is the announced result about the extinction probabilities $q_{x,y}$ for these two models. Its proof is provided in the final section of this article.

\begin{proposition}\label{prop:extinction IBCOS BUTS}
Given any of the $2$-type population models \IBCOS\ or \BUTS, let $(X_{n},Y_{n})_{n\ge 0}$ be the random walk describing the subpopulation sizes for the two types and $q_{x,y}$ the probability of absorption at one of the axes given $X_{0}=x$ and $Y_{0}=y$. Then
$$ q_{x,y}\ =\ \mu^{x}+\mu^{y}-\mu^{x+y} $$
for all $x,y\in\N_{0}$.
\end{proposition}

\subsection{Related work and further generalizations}\label{subsec:related}

\unitlength=1.3cm
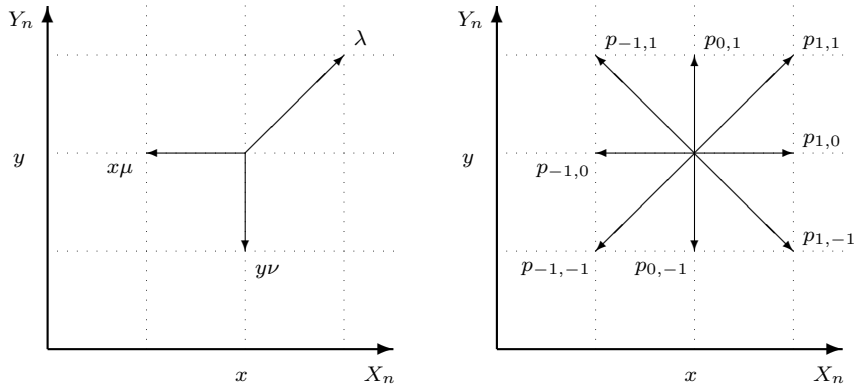
\begin{figure}[t!]
\centering 
\hspace{-1.8cm} 
 \begin{tabular}{cc}

    \begin{picture}(3.3,4.5)
    \thicklines
    \put(1,1){{\vector(1,0){3.5}}}
    \put(1,1){\vector(0,1){3.5}}
    \put(4.2,0.68){$X_{n}$}
    \put(0.6,4.3){$Y_{n}$}
    \put(0.65,2.9){$y$}
    \put(2.9,0.68){$x$}
    \thinlines
    \put(3,3){\vector(-1,0){1}}
    \put(3,3){\vector(1,1){1}}
    \put(3,3){\vector(0,-1){1}}
    \put(1.6,2.8){$x\mu$}
    \put(3.1,1.77){$y\nu$}    
    \put(4.1,4.1){$\lambda$}
   \linethickness{0.1mm}
    \put(1,2){\dottedline{0.1}(0,0)(3.5,0)}
    \put(1,3){\dottedline{0.1}(0,0)(3.5,0)}
    \put(1,4){\dottedline{0.1}(0,0)(3.5,0)}
    \put(2,1){\dottedline{0.1}(0,0)(0,3.5)}
    \put(3,1){\dottedline{0.1}(0,0)(0,3.5)}
    \put(4,1){\dottedline{0.1}(0,0)(0,3.5)}
    \end{picture}\hspace{1.2cm}
&   \begin{picture}(4,4)
    \thicklines
    \put(1,1){{\vector(1,0){3.5}}}
    \put(1,1){\vector(0,1){3.5}}
    \put(4.2,0.68){$X_{n}$}
    \put(0.6,4.3){$Y_{n}$}
    \put(0.65,2.9){$y$}
    \put(2.9,0.68){$x$}
    \thinlines
    \put(3,3){\vector(1,0){1}}
    \put(3,3){\vector(-1,0){1}}
    \put(3,3){\vector(0,1){1}}
    \put(3,3){\vector(0,-1){1}}
    \put(3,3){\vector(1,1){1}}
    \put(3,3){\vector(-1,1){1}}
    \put(3,3){\vector(1,-1){1}}
    \put(3,3){\vector(-1,-1){1}}
    \put(4.1,4.1){$p_{1,1}$}
    \put(4.1,3.1){$p_{1,0}$}
    \put(1.4,2.8){$p_{-1,0}$}
     \put(1.25,1.8){$p_{-1,-1}$}
     \put(2.4,1.8){$p_{0,-1}$}
    \put(4.1,2.1){$p_{1,-1}$}
    \put(3.1,4.1){$p_{0,1}$}
    \put(2.1,4.1){$p_{-1,1}$}
   \linethickness{0.1mm}
    \put(1,2){\dottedline{0.1}(0,0)(3.5,0)}
    \put(1,3){\dottedline{0.1}(0,0)(3.5,0)}
    \put(1,4){\dottedline{0.1}(0,0)(3.5,0)}
    \put(2,1){\dottedline{0.1}(0,0)(0,3.5)}
    \put(3,1){\dottedline{0.1}(0,0)(0,3.5)}
    \put(4,1){\dottedline{0.1}(0,0)(0,3.5)}
    \end{picture}
    \end{tabular}
\vspace{-0.8cm}
\caption{Transition probabilities at $(x,y)$ for $(X_{n},Y_{n})_{n\ge 0}$ in the M/M/$\infty$ parallel queuing model with simultaneous arrivals of \cite{Foddy:84} (left panel) and for a generalization to eight neighbors of the \BUTS\ model (right panel). See also Tab.~\ref{tab:different_models}.}
\label{fig:Foddy}
\end{figure}

We conclude this section by mentioning some related work. First, the  probabilities of absorption at a given axis for the \BUTS\ model are computed in Thm.~13 by \cite{KurkovaRaschel:11} in terms of integrals of Chebychev polynomials; applications of these results in finance (study of Markovian order books) can be found in \cite{ContLarrard:13}, see in particular Prop.~3 there. 

For a generalization of the \BUTS\ model with eight jump directions (see the right panel of Fig.~\ref{fig:Foddy} as well as Tab.~\ref{tab:different_models}), it is shown by \cite{KurkovaRaschel:11} (in the presence of a positive drift, see Prop.~9 there) that $\frac{a_{x,y}}{2}\leq q_{x,y}\leq a_{x,y}$, where
\begin{equation*}
     a_{x,y}=\left(\frac{p_{-1,-1}+p_{0,-1}+p_{1,-1}}{p_{-1,1}+p_{0,1}+p_{1,1}}\right)^{x}+\left(\frac{p_{-1,-1}+p_{-1,0}+p_{-1,1}}{p_{1,-1}+p_{1,0}+p_{1,1}}\right)^{y}.
\end{equation*}

Second, when restricting our model (with jumps as in Fig.~\ref{fig:transition rates}) to the very particular case $\lambda=1$ (pure death model), i.e., with no jumps to the North and the East, the probabilities of absorption at a given axis are computed by \cite{ErnstGrigorescu:17}, together with a proposed interpretation within the framework of a war-of-attrition problem. 

Finally, for a queueing model with South and West rates as in Fig.~\ref{fig:transition rates}, and one additional homogeneous North-East rate (see Fig.~\ref{fig:Foddy} and Tab.~\ref{tab:different_models}), Foddy in her PhD thesis arrives at a closed-form expression for the generating function of the stationary distribution, see Thm.~24 by \cite{Foddy:84}.

\section{Results}\label{sec:results}

\subsection{Statement of main results}

The following functions on $\N_{0}^{2}$ will appear in our results and frequently be used in our analysis, namely
\begin{gather}
\label{eq1:various_functions_important}
f_{0}(x,y)\,:=\,\mu^{x+y},\quad f_{1}(x,y)\,:=\,\mu^{x}+\mu^{y}
\shortintertext{and}
\label{eq2:various_functions_important}
h(x,y)\,:=\,\mu^{x}+\mu^{y}-\mu^{x+y}\,=\,f_{1}(x,y)-f_{0}(x,y).
\end{gather}
The very same functions divided by the binomial coefficient ${\binom{x+y}{x}}=\frac{(x+y)!}{x! y!}$ are denoted $\wh{f}_{0}$, $\wh{f}_{1}$ and $\wh{h}$, respectively. Our first theorem restates inequalities \eqref{eq:basic inequality q(x,y)} and \eqref{eq:first bound kappa} with improved lower bounds.

\begin{theorem}\label{thm:annealed inequality}
For all $x,y\in\N_{0}$,
\begin{equation}\label{eq:improved bounds q_xy}
f_{0}(x,y)\vee\left[\left(1+\frac{\mu}{2(1+\mu)}\right)^{x\wedge y}\wh{f}_{1}(x,y)\right]\ \le\ q_{x,y}\ \le\ h(x,y).
\end{equation}
As a consequence,
\begin{equation}\label{eq:improved bounds kappa}
\mu^{2}\vee\left[\frac{\mu}{4}\left(1+\frac{\mu}{2(1+\mu)}\right)\right]\ \le\ \kappa_{*}\ \le\ \kappa^{*}\ \le\ \mu.
\end{equation}
\end{theorem}

\begin{figure}[ht!]
\begin{center}
\includegraphics[height=4.5cm]{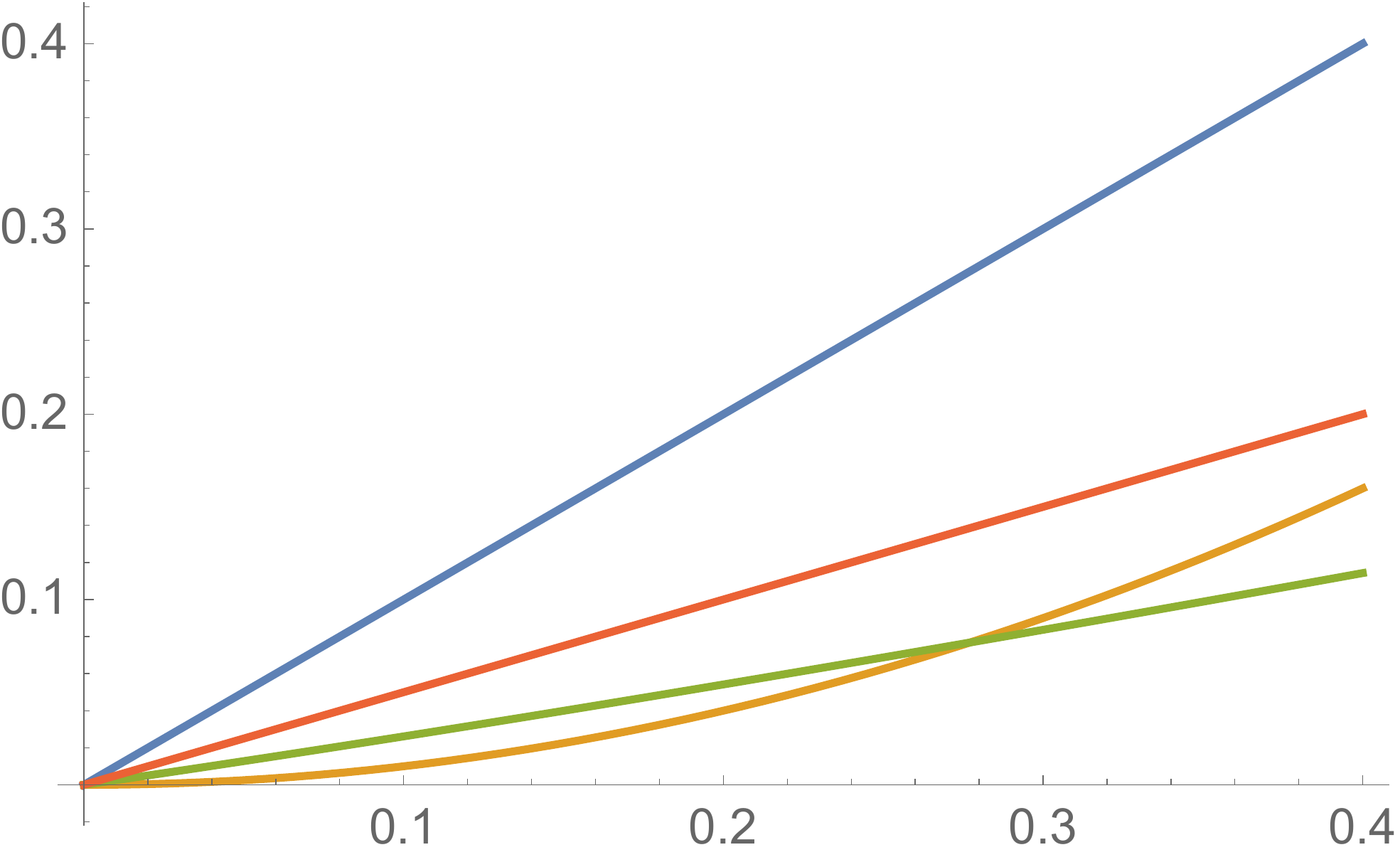} 
\end{center}
\vspace{-.5cm}
\hspace{6.7cm}$\mu$
\vspace{.3cm} 
\caption{The curves $\mu\mapsto\mu^{2}$ (orange) and $\mu\mapsto\frac{\mu}{4}(1+\frac{\mu}{2(1+\mu)})$ (green) for small values of $\mu$. The identity function and half the identity are shown in blue and red, respectively.}
\label{fig:lower bound q(x,y)}
\end{figure}

\begin{figure}[b!]
\begin{center}
\includegraphics[width=7.8cm,height=6.2cm]{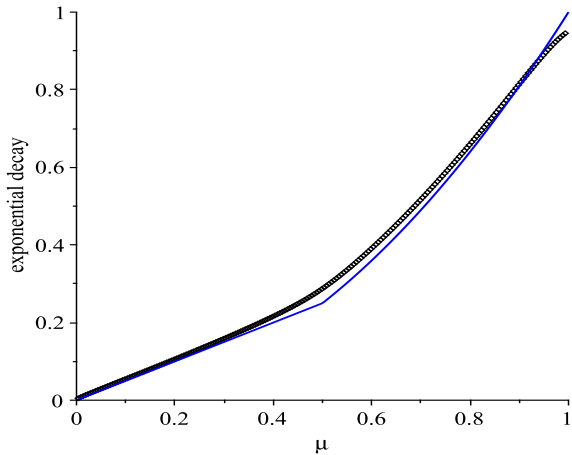} 
\end{center}  
\caption{Simulation results, realized with R, of the exponential decay $\kappa$ defined in \eqref{eq:rate of decay of q_xx} (black) and comparison with the conjectured curve \eqref{eq:conjecture_kappa} (blue). For each $\mu=\frac{k}{200}$, $k=1,\ldots,199$, we did 1000 simulation runs of the population starting from $(x,x)=(100,100)$ for $n=1,\ldots,1000$. Then our approximation for the exponential decay is simply the $\frac{1}{x}$th power of the proportion of populations which did not survive within the chosen time interval.}
\label{fig:Alsmeyer-Raschel-conjecture}
\end{figure}

\begin{figure}[ht]
\begin{center}
\includegraphics[width=3.5cm,height=3.5cm]{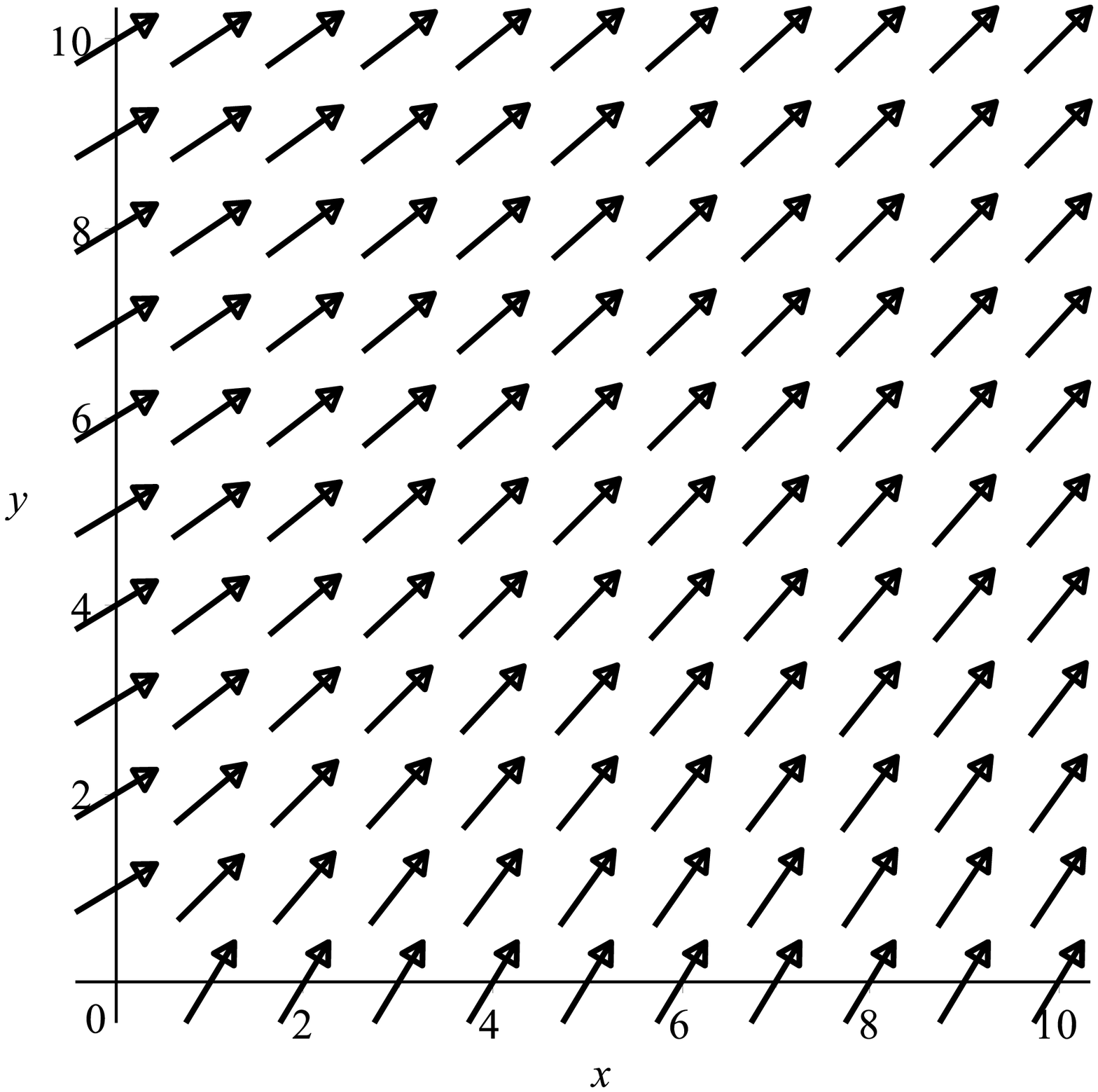} 
\quad
\includegraphics[width=3.5cm,height=3.5cm]{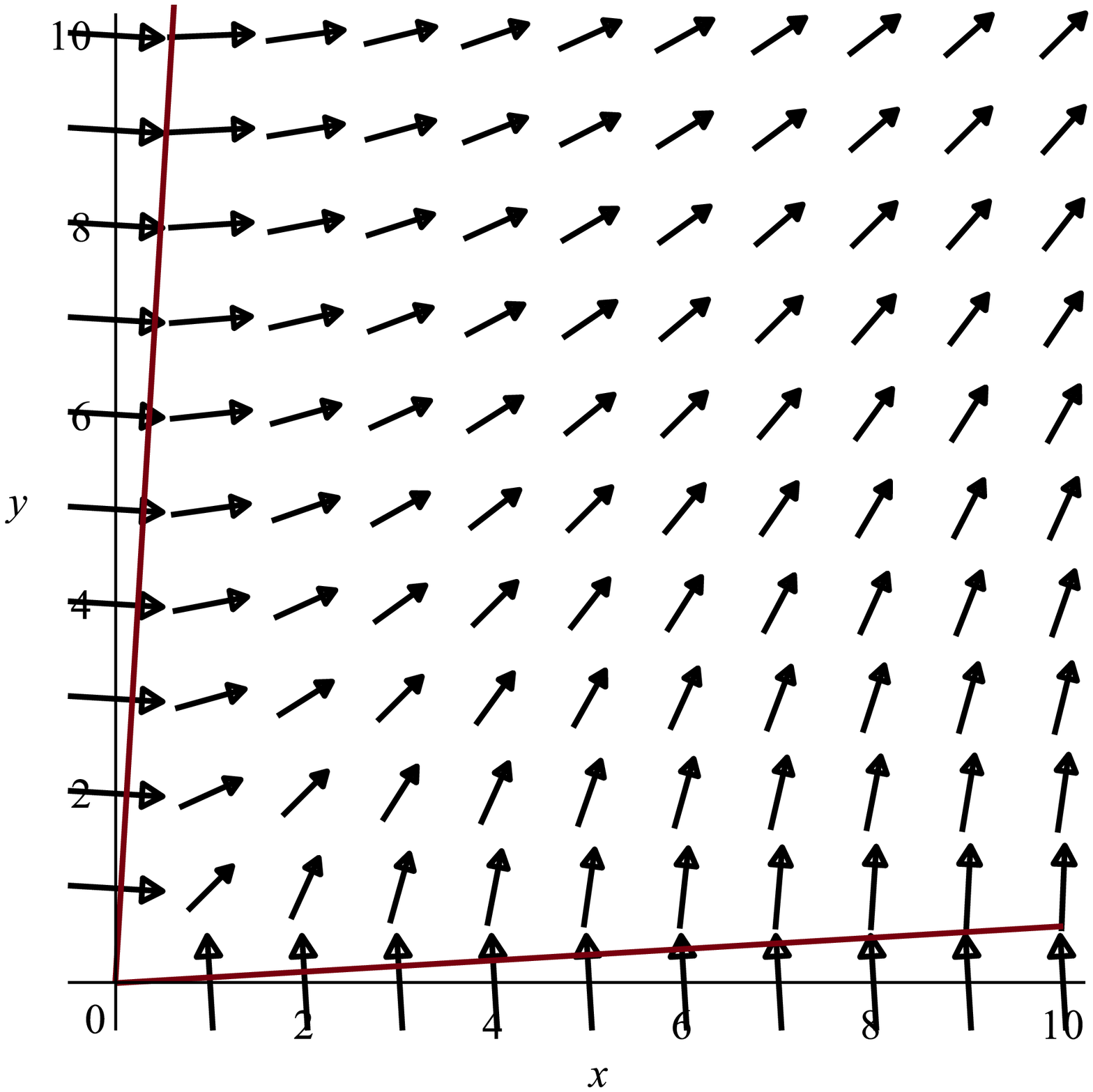} 
\quad
\includegraphics[width=3.5cm,height=3.5cm]{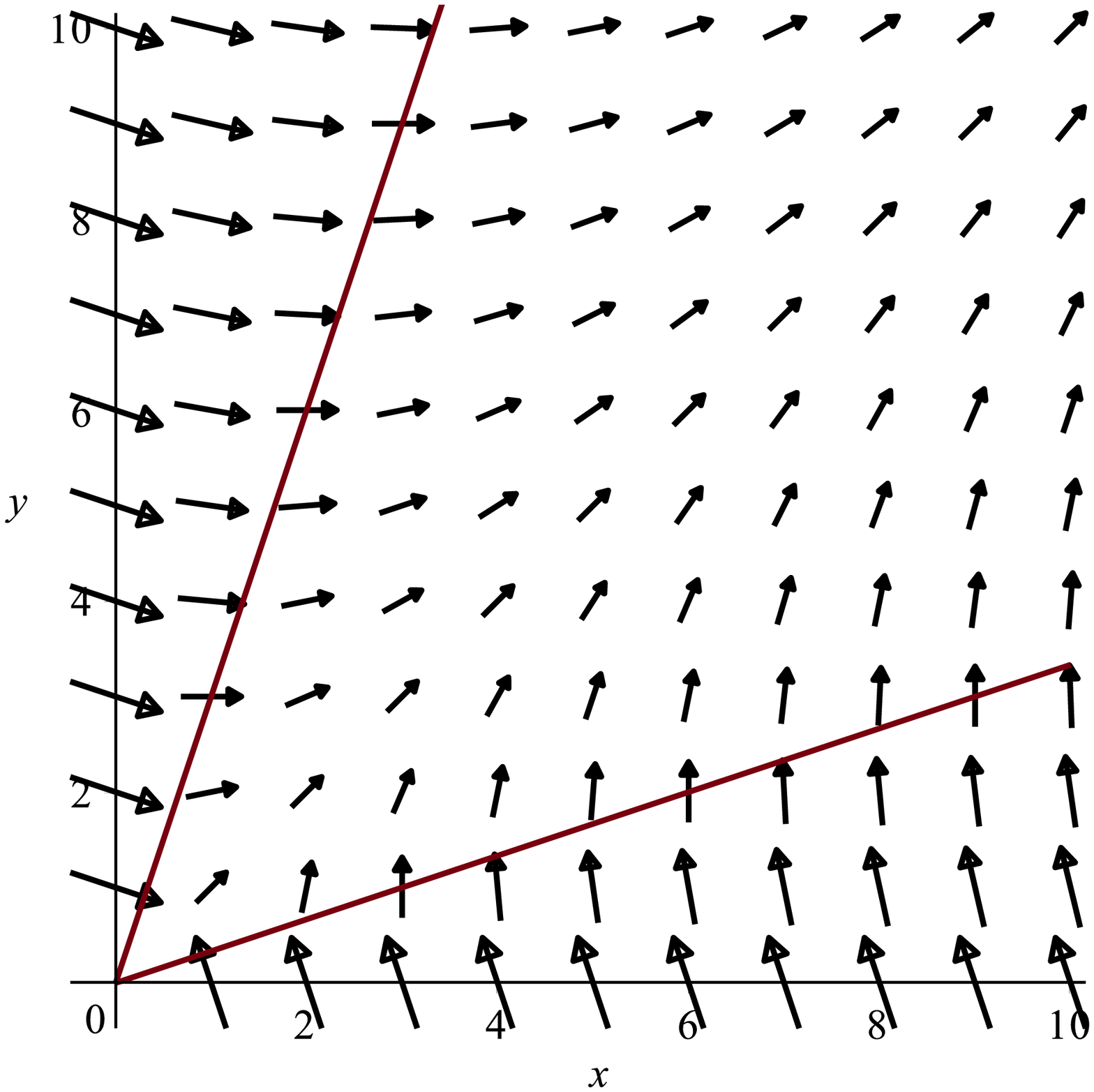} 
\end{center}  
\caption{Drift vectors for $\mu=0.2$ (left), $\mu=0.53$ (middle), and $\mu=0.66$ (right). The red lines are $y=(2\mu-1)x$ and $y=(2\mu-1)^{-1}x$. They are located in the positive quarter plane if, and only if, $\mu\in[\frac{1}{2},1]$. Within the cone delimitated by these two lines, both drift coordinates are positive, while outside the cone one of them becomes negative.}
\label{fig:vector_fields}
\end{figure}

It will be shown in Sec.~\ref{sec:proof theorem 1} that the functions $h$ and $\wh{h}$ are super- and subharmonic for $P$, respectively, with $h(x,y)=\wh{h}(x,y)=1$ if $x\wedge y=0$. This leads to
\begin{equation}\label{eq:natural inequality q_xy}
\wh{h}(x,y)\ \le\ q_{x,y}\ \le\ h(x,y)
\end{equation}
for all $x,y\in\N_{0}$ as an almost direct consequence (use Lem.~\ref{lem1}). But since 
$$ \wh{h}(x,y)\ =\ \wh{f}_{1}(x,y)\,-\,\wh{f}_{0}(x,y)\ \le\ \Big(1+\frac{\mu}{2(1+\mu)}\Big)^{x\wedge y}\wh{f}_{1}(x,y), $$
we see that inequality \eqref{eq:improved bounds q_xy} stated in our Thm.~\ref{thm:annealed inequality} is stronger. Its lower bound does indeed provide a strong improvement over that in \eqref{eq:first bound kappa} for small values of the parameter $\mu$ (see Fig.~\ref{fig:lower bound q(x,y)}). On the other hand, our conjecture, supported by Fig.~\ref{fig:Alsmeyer-Raschel-conjecture}, is that $\kappa_{*}=\kappa^{*}=\kappa$ and
\begin{equation}\label{eq:conjecture_kappa}
          \kappa=\left\{\begin{array}{ll}
          \displaystyle\frac{\mu}{2} & \text{if } \mu\in[0,\frac{1}{2}),\smallskip\\
          \mu^{2} & \text{if } \mu\in[\frac{1}{2},1],
          \end{array}\right.
\end{equation}     
and that a phase transition occurs at $\mu=\frac{1}{2}$. Further evidence in support of \eqref{eq:conjecture_kappa} or at least of the higher rate $\kappa=\mu^{2}$ for $\mu>\frac{1}{2}$ is provided by the following argument: for such $\mu$, there are regions, confined by an axis and a neighboring red line in Fig.~\ref{fig:vector_fields}, where the drift vector $(\Erw_{x,y}X_{1}-x,\Erw_{x,y}Y_{1}-y)$ has a negative component or, to be more precise, $\Erw_{x,y}(X_{1}\vee Y_{1})<x\vee y$ holds. This pushes the walk closer to the origin and we conjecture 
that this effect is strong enough to imply
$$ \log\kappa\ =\ \lim_{x\to\infty}x^{-1}\log\Prob_{x,x}(\tau<\infty)\ =\ \lim_{x\to\infty}x^{-1}\log\Prob_{x,x}(\vth_{x^{\eps}}<\infty) $$
for any $0<\eps<1$, where $\vth_{c}:=\inf\{n\geq0:X_{n}+Y_{n}\le c\}$ for $c\ge 0$. But then it can be shown, at least  for sufficiently small $\eps$, that the second limit equals $\log\mu^{2}$ (and so $\kappa=\mu^{2}$) when observing that in the event $\vth_{x^{\eps}}<\infty$ at least $2x-\lfloor x^{\eps}\rfloor$ of the independent subpopulations stemming from one of the $2x$ ancestors of the whole population must die out. Namely, the probability for this to happen equals
$$ \sum_{k=2x-\lfloor x^{\eps}\rfloor}^{2x}\binom{2x}{k}\mu^{k}(1-\mu)^{2x-k} $$
because $\mu$ is the extinction probability for any of these subpopulations and thus $\binom{2x}{k}\mu^{k}(1-\mu)^{2x-k}$ the probability that exactly $k$ subpopulations die out and the other $2x-k$ survive. By making use of Stirling's formula to bound the binomial coefficients and some straightforward estimation of the sum, the result then follows. Further details are omitted. Unfortunately, we have no intuitive explanation for the rate $\frac{\mu}{2}$ in the case $\mu\in[0,\frac{1}{2}]$.

\vspace{.2cm}
The following corollary is a straightforward consequence of \eqref{eq:natural inequality q_xy}, more precisely of $q_{x,y}\ge\wh{h}(x,y)$, when using the standard asymptotics for $\binom{x+y}{x}$.

\begin{corollary}\label{cor:estimate_boundary}
For any fixed $y\in\N$ and $x\to\infty$, 
\begin{equation*}
\liminf_{x\to\infty}x^{y}q_{x,y}\ \ge\ y!\,\mu^{y}.
\end{equation*}     
\end{corollary}

The result should be compared with the exact asymptotics stated in \eqref{eq:q_xy for fixed y}, the only difference there being an extra term $2^{y}$ which does not vary with $x$.

\vspace{.05cm}
The most difficult of our theorems is next and asserts that the exponential decay of the extinction probability $q_{x,x}$ is strictly less than $\mu$.

\begin{theorem}\label{thm:q_x,y asymptotic}
If $\mu<1$, then there exists $\nu\in (0,\mu)$ such that
\begin{equation*}
q_{x,x}\,=\,o\left(\nu^{x}\right)
\end{equation*}
as $x\to\infty$, thus $\kappa^{*}<\mu$.
\end{theorem}

\subsection{Interpretation of the results}

Let us first list a few reasons which have motivated our analysis of the exponential decay of the extinction probability when initial subpopulation sizes tend to infinity. Recall that extinction here means the disappearance of one of these subpopulations (types) rather than of the whole population.
\begin{itemize}\itemsep4pt
     \item As shown in Thm.~\ref{thm:annealed inequality} and Thm.~\ref{thm:q_x,y asymptotic}, the extinction probability $q_{x,x}$ decreases to $0$ exponentially fast in the sense that $\kappa^{*}=\limsup_{x\to\infty}q_{x,x}^{1/x}<1$. The exponential decay rate therefore provides a first approximation of the probability of this rare event of extinction.
     \item Since an exact formula for $q_{x,x}$ seems to be difficult to come by, any information on the decay rate is relevant and insightful.
     \item In statistical mechanics, the rate of exponential decay or growth (also called the connective constant) typically carries combinatorial and probabilistic information, as for example the limiting free energy, see for instance the book by \cite{Baxter}.
\end{itemize}
\unitlength=1.2cm
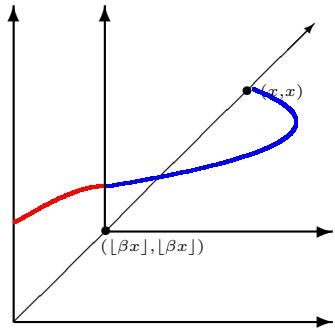
\begin{figure}[b]
\centering 
    \begin{picture}(5,4.5)
    \thicklines
    \put(1,1){\vector(1,0){3.5}}
    \put(2,2){\vector(1,0){2.5}}
    \put(1,1){\vector(0,1){3.5}}
    \put(2,2){\vector(0,1){2.5}}
    \thinlines
   \linethickness{0.2mm}
    \put(1,1){\vector(1,1){3.3}}
    \put(1.95,1.95){$\bullet$}
    \put(1.95,1.78){$\scriptstyle(\lfloor\beta x\rfloor,\lfloor\beta x\rfloor)$}
    \put(3.7,3.5){$\scriptstyle (x,x)$}
    \put(3.5,3.5){$\bullet$}
    \linethickness{0.5mm}
    \textcolor{blue}{\qbezier(3.55,3.57)(5,3)(1.95,2.5)}
    \textcolor{red}{\qbezier(1.8,2.5)(1.5,2.5)(0.85,2.1)}
    \end{picture}\vspace{-0.8cm}
\caption{In the event of absorption at one of the axes, the walk $(X_{n},Y_{n})_{n\ge 0}$, when starting at $(x,x)$, must necessarily pass through one of the halflines emanating from $(\lfloor\beta x\rfloor,\lfloor\beta x\rfloor)$ for $\beta\in (0,1)$. A typical absorbed trajectory can thus be split in two parts: a first part prior to $\tau_{\beta x}$ (in blue) and a final part (in red) which in fact may return to the inner cone before absorption.}
\label{fig:subplane}
\end{figure}

\noindent
Some more biologically oriented interpretations of our results are next.
 
\begin{itemize}\itemsep4pt
\item By Thm.~\ref{thm:q_x,y asymptotic}, the exponential decay rate of $q_{x,x}$ is strictly smaller than $\mu$, the corresponding rate in both \BUTS\ and \IBCOS. The fact that our hybrid model can be obtained from \IBCOS\ by replacing the type selection probabilities in the event of a death with those of \BUTS, thus $(\frac{1}{2},\frac{1}{2})$ with $(\frac{x}{x+y},\frac{y}{x+y})$, explains that $q_{x,y}$ must be smaller than the corresponding probability for \IBCOS\ (formally confirmed for any $q_{x,y}$ by Prop.~\ref{prop:extinction IBCOS BUTS} and Thm.~\ref{thm:annealed inequality}). Namely, the death of an individual is now more likely to occur in the larger subpopulation. However, it is a priori not clear at all and therefore the most interesting information of Thm.~\ref{thm:q_x,y asymptotic} that this effect is strong enough to even show on a logarithmic scale by having $\limsup_{x\to\infty}\frac{1}{x}\log q_{x,x}<\log\mu$. An intuitive explanation the authors admittedly do not have would certainly be appreciated.
    
\item By Thm.~\ref{thm:annealed inequality}, $\mu^{2}$ constitutes a lower bound for the exponential decay rate of $q_{x,x}$ which is strict at least for small values of $\mu$, see Fig.~\ref{fig:lower bound q(x,y)}. Since $\mu^{2}$ also equals the decay rate of the probability of extinction of the whole population in the modified model described in Subsect.~\ref{subsec:2-type splitting}, this indicates that if one of the subpopulations (types) dies out in this modification, then the other one may well survive.
\end{itemize}

\noindent   
Let us finally give a sketch of the main arguments to prove Thm.~\ref{thm:q_x,y asymptotic}, which is both the most interesting and most difficult result of this article. 

The main idea for obtaining the result is based on a path splitting illustrated by Fig.~\ref{fig:subplane}. Starting from a remote diagonal point $(x,x)$, the random walk must first hit at least once, say at $\tau_{\beta x}$, the boundary of the quarter plane with edge point $(\lfloor\beta x\rfloor,\lfloor\beta x\rfloor)$ for an arbitrarily fixed $\beta\in (0,1)$. By using the strong Markov property, this leads to a factorization of $q_{x,y}$ into the probability that $\tau_{\beta x}$ is finite times, roughly speaking, the maximal probability of absorption when starting at a point $(x',y')$ with $x'\wedge y'= \lfloor\beta x\rfloor$ and $|x'-y'|\ge\alpha x$ for some $\alpha>0$. The latter can be inferred by showing that $|X_{\tau_{\beta x}}-Y_{\tau_{\beta x}}|\ge\alpha x$ with very high probability as $x\to\infty$. Finally, it is proved that this maximal probability of absorption is bounded by a constant times $\nu^{\beta x}$ for some $\nu<\mu$ which then easily leads to the conclusion of the theorem.

\subsection{Two additional results}
It is intuitively appealing and the last theorem therefore more a reassurance than a surprise that $q_{x,y}$ decreases in both arguments. On the other hand, the proof requires some care and will be based on a coupling argument.


\begin{theorem}\label{thm:q(x,y) monotonic}
As a function of $(x,y)\in\N^{2}$, $q_{x,y}$ is nonincreasing in each argument, thus
$$ q_{x,1}\,\le\,q_{x,2}\,\le\,\ldots\,q_{x,x}\,\le\,q_{x,x+1}\,\le\ldots $$
for all $x\in\N$.
\end{theorem}

In connection with the later use of $P$-sub- and $P$-super\-harmonic functions in order to prove Thm.~\ref{thm:annealed inequality}, a required property of our random walk is the following \emph{standard behavior}: If $(X_{n},Y_{n})_{n\ge 0}$ is not absorbed at one of the axes $(\tau=\infty)$, then explosion not only of the total population size occurs, i.e., $X_{n}+Y_{n}\to\infty$ a.s., but actually of both subpopulation sizes as well, i.e., $X_{n}\wedge Y_{n}\to\infty$ a.s., see Lem.~\ref{lem1} for the result where this property enters.
The property is stated in the subsequent proposition and proved at the end of Sec.~\ref{sec:proof theorem 1}.

\begin{proposition}\label{prop:standard}
The random walk $(X_{n},Y_{n})_{n\ge 0}$ on $\N_{0}^{2}$ with transition probabilities defined by \eqref{eq:transition probabilities} on $\N^{2}$ exhibits standard behavior in the sense that
\begin{equation}\label{eq:standard behavior}
\Prob_{x,y}\left(\,\lim_{n\to\infty}X_{n}\wedge Y_{n}=\infty\,\Big|\,\tau=\infty\right)\ =\ 1
\end{equation}
for all $x,y\in\N$.
\end{proposition}

Let us finally stipulate for the rest of this work that $\Prob$ will be used for probabilities that are the same under any $\Prob_{x,y}$, $x,y\in\N_{0}$.

\section{Annealed harmonic analysis and proof of Thm.~\ref{thm:annealed inequality}}\label{sec:proof theorem 1}

The purpose of this section is to show that standard harmonic analysis for the transition operator $P$ of $(X_{n},Y_{n})_{n\ge 0}$, see \eqref{eq:transition operator P}, provides the appropriate tool to rather easily derive bounds for $q_{x,y}$ of the form stated in Thm.~\ref{thm:annealed inequality}.

\subsection{Sub(super)harmonic functions and applications}
\label{subsec:expression_harmonic}

We begin with the statement of the following basic result, valid for any transition operator $P$ on $\N_{0}^{2}$ that is absorbing on the axes.

\begin{lemma}\label{lem1}
Let $f$ be any nonnegative sub$[$super$]$harmonic function for $P$ such that $f(x,y)=1$ for $x\wedge y=0$ and $\lim_{x\wedge y\to\infty}f(x,y)=0$. Then $q_{x,y}\ge[\le]f(x,y)$ for all $x,y\in\N_{0}$.
\end{lemma}

\begin{proof}
By the assumptions, $(f(X_{n},Y_{n}))_{n\ge 0}$ forms a nonnegative bounded submartingale [supermartingale] satisfying $f(X_{\tau},Y_{\tau})=1$ on $\{\tau<\infty\}$. Moreover, by Prop.~\ref{prop:standard}, $f(X_{n},Y_{n})\to 0$ a.s.\ on $\{\tau=\infty\}$ so that $f(X_{\tau\wedge n},Y_{\tau\wedge n})$ converges a.s.\ to $\1_{\{\tau<\infty\}}$. Now use the Optional Sampling Theorem to infer
\begin{equation*}
f(x,y)\ \le[\ge]\ \lim_{n\to\infty}\Erw_{x,y}f(X_{\tau\wedge n},Y_{\tau\wedge n})\ =\ \Prob_{x,y}(\tau<\infty)\ =\ q_{x,y}
\end{equation*}
for all $x,y\in\N_{0}$.\qed
\end{proof}

Besides $f_{0}$, $f_{1}$ and $h$ already defined at the beginning of Sec.~\ref{sec:results}, the following functions will also be useful hereafter:
\begin{align}
\begin{split}\label{eq:various_functions}
&f_{\wedge}(x,y)\,:=\,\mu^{x\wedge y},\\
&f_{\vee}(x,y)\,:=\,\mu^{x\vee y},\\
&f_{2}(x,y)\,:=\,f_{\wedge}(x,y)-f_{\vee}(x,y)\ =\ \mu^{x\wedge y}-\mu^{x\vee y}.
\end{split}
\end{align}
Their counterparts multiplied with $\binom{x+y}{x}^{-1}$ are denoted $\wh{f}_{\wedge}$, $\wh{f}_{\vee}$ and $\wh{f}_{2}$. In order to prove Thm.~\ref{thm:annealed inequality}, we continue with a derivation of the $P$-harmonic properties of these functions. The results particularly show that $h$ is $P$-superharmonic (Lem.~\ref{lem5}), and that $\wh{f}_{1}$ and $\wh{h}$ are $P$-subharmonic (Lem.~\ref{lem6} and \ref{lem8}). 

\begin{lemma}\label{lem2}
The function $f_{0}(x,y)=\mu^{x+y}$ is $P$-harmonic.
\end{lemma}

\begin{proof}
For all $x,y\in\N$, we have that
\begin{equation*}
Pf_{0}(x,y)\ =\ \lambda\mu^{x+y-1}+\ovl{\lambda}\mu^{x+y+1}\ =\ \ovl{\lambda}\mu^{x+y}+\lambda\mu^{x+y}\ =\ \mu^{x+y},
\end{equation*}
which proves the assertion.\qed
\end{proof}

The subsequent lemmata will provide formulae for $Pf_{\wedge}$, $Pf_{\vee}$, etc.

\begin{lemma}\label{lem3}
For $x,y\in\N$, the function $f_{\wedge}(x,y)$ satisfies
\begin{align*}
Pf_{\wedge}(x,y)\ =\ 
\begin{cases}
\hfill 2\ovl{\lambda}\,f_{\wedge}(x,y),&\text{if }x=y,\\[2mm]
\displaystyle\left(1-(1-2\lambda)\,\frac{|x-y|}{2(x+y)}\right)f_{\wedge}(x,y),&\text{if }x\ne y.
\end{cases}
\end{align*}
As a consequence,
\begin{align*}
Pf_{\wedge}(x,y)\ 
\begin{cases}
\ge\ f_{\wedge}(x,y),&\text{if }x=y,\\
\le\ f_{\wedge}(x,y),&\text{if }x\ne y.
\end{cases}
\end{align*}
\end{lemma}

\begin{proof}
It suffices to consider $x\le y$. Note that $\lambda<1-\lambda$ implies $\ovl{\lambda}x+\lambda y\le\frac{1}{2}(x+y)$. Using this, we find for $x<y$
\begin{align*}
Pf_{\wedge}(x,y)\ &=\ \lambda\left(\frac{x}{x+y}\,\mu^{x-1}\,+\,\frac{y}{x+y}\,\mu^{x}\right)\,+\,\ovl{\lambda}\,\frac{\mu^{x+1}+\mu^{x}}{2}\\
&=\ \mu^{x}\left(\ovl{\lambda}\,\frac{x}{x+y}\,+\,\lambda\,\frac{y}{x+y}\,+\,\frac{1}{2}\right)\\
&=\ \mu^{x}\left(1\,-\,(1-2\lambda)\frac{|x-y|}{2(x+y)}\right)\\
&\le\ \mu^{x}\ =\ f_{\wedge}(x,y),
\end{align*}
whereas for $x=y$,
\begin{equation*}
Pf_{\wedge}(x,x)\ =\ \lambda\mu^{x-1}+\ovl{\lambda}\mu^{x}\ =\ 2\ovl{\lambda}\mu^{x}\ \ge\ \mu^{x}\ =\ f_{\wedge}(x,x)
\end{equation*}
holds true as claimed.\qed
\end{proof}

\begin{lemma}\label{lem4}
For $x,y\in\N$, the function $f_{\vee}(x,y)$ satisfies
\begin{align*}
Pf_{\vee}(x,y)\ =\ 
\begin{cases}
\hfill 2\lambda f_{\vee}(x,y),&\text{if }x=y,\\[2mm]
\displaystyle\left(1+(1-2\lambda)\frac{|x-y|}{2(x+y)}\right)f_{\vee}(x,y),&\text{if }x\ne y.
\end{cases}
\end{align*}
As a consequence,
\begin{align*}
Pf_{\vee}(x,y)\ 
\begin{cases}
\le\ f_{\vee}(x,y),&\text{if }x=y,\\
\ge\ f_{\vee}(x,y),&\text{if }x\ne y.
\end{cases}
\end{align*}
\end{lemma}

\begin{proof}
Again, it suffices to consider $x\le y$. Note that $\lambda<1-\lambda$ then implies $\lambda x+\ovl{\lambda}y\ge\frac{1}{2}(x+y)$. For $x<y$, we obtain
\begin{align*}
Pf_{\vee}(x,y)\ &=\ \lambda\left(\frac{x}{x+y}\,\mu^{y}\,+\,\frac{y}{x+y}\,\mu^{y-1}\right)+\ovl{\lambda}\frac{\mu^{y}+\mu^{y+1}}{2}\\
&=\ \mu^{y}\left(\lambda\frac{x}{x+y}\,+\,\ovl{\lambda}\frac{y}{x+y}\,+\,\frac{1}{2}\right)\\
&=\ \mu^{y}\left(1+(1-2\lambda)\frac{|x-y|}{2(x+y)}\right)\\
&\ge\ \mu^{y}\ =\ f_{\vee}(x,y),
\end{align*}
where $\lambda x+\ovl{\lambda}y\ge\frac{1}{2}(x+y)$ has been utilized. If $x=y$, then
\begin{equation*}
Pf_{\vee}(x,x)\ =\ \lambda\mu^{x}+\ovl{\lambda}\mu^{x+1}\ =\ 2\lambda\mu^{x}\ \le\ \mu^{x}\ =\ f_{\vee}(x,x),
\end{equation*}
which completes the proof.\qed
\end{proof}

\begin{lemma}\label{lem5}
For $x,y\in\N$, the function $h$ in \eqref{eq2:various_functions_important} satisfies
\begin{align}\label{eq:1st iterate of h}
Ph(x,y)\ =\ h(x,y)\,-\,(1-2\lambda)\frac{\vert x-y\vert}{2(x+y)}f_{2}(x,y)\ \le\ h(x,y)
\end{align}
with $f_{2}$ defined in \eqref{eq:various_functions}, and is thus $P$-superharmonic. Furthermore, the same identity holds true for $f_{1}=h-f_{0}$, and
\begin{align}\label{eq:1st iterate of g_2}
Pf_{2}(x,y)\ =\ 
\begin{cases}
\hfill (1-2\lambda)f_{1}(x,y),&\text{if }x=y,\\[2mm]
\displaystyle f_{2}(x,y)-(1-2\lambda)\frac{|x-y|}{2(x+y)}f_{1}(x,y),&\text{if }x\ne y.
\end{cases}
\end{align}
\end{lemma}

\begin{proof}
It suffices to prove \eqref{eq:1st iterate of h} for $f_{1}$, for $f_{0}=f_{1}-h$ is harmonic by Lem.~\ref{lem2}. With the help of Lem.~\ref{lem3} and \ref{lem4}, we obtain for $x<y$
\begin{align*} 
Pf_{1}(x,y)\ &=\ Pf_{\wedge}(x,y)\,+\,Pf_{\vee}(x,y)\\
&=\ f_{\wedge}(x,y)\,+\,f_{\vee}(x,y)\,-\,(1-2\lambda)\frac{|x-y|}{2(x+y)}(f_{\wedge}(x,y)-f_{\vee}(x,y))\\
&=\ f_{1}(x,y)\,-\,(1-2\lambda)\frac{y-x}{2(x+y)}f_{2}(x,y)\\
&\le\ f_{1}(x,y),
\end{align*}
and for $x=y\ ($note that obviously $f_{\wedge}(x,x)=f_{\vee}(x,x))$
\begin{equation*}
Pf_{1}(x,x)\ =\ (2\ovl{\lambda}+2\lambda)f_{\wedge}(x,x)\ =\ 2f_{\wedge}(x,x)\ =\ f_{1}(x,x),
\end{equation*}
which proves \eqref{eq:1st iterate of h}. Eq.\ \eqref{eq:1st iterate of g_2} follows in a similar manner.\qed
\end{proof}

Turning to the harmonic properties of $\wh{f}_{0}$, $\wh{f}_{1}$ and $\wh{h}$, we first study $\wh{f}_{1}$.

\begin{lemma}\label{lem6}
The function $\wh{f}_{1}$ is subharmonic for $P$, in fact for $x,y\in\N$
\begin{align*}
&P\wh{f}_{1}(x,y)-\wh{f}_{1}(x,y)\\  
&\hspace{1cm}\left.\begin{cases}
=\ \displaystyle\frac{1}{2}\,\left(1+\frac{1}{x+y+1}\right)\wh{f}_{1}(x,y),&\text{if }x=y\\[3mm]
\ge\ \displaystyle\frac{\lambda}{2}\,\left(1+\frac{1}{x+y+1}\right)\wh{f}_{1}(x,y),&\text{if }x\ne y
\end{cases}\right\}\ \ge\ \frac{\lambda}{2}\,\wh{f}_{1}(x,y).
\end{align*}
\end{lemma}
 
\begin{proof}
Using
\begin{align*}
\frac{x}{x+y}\,\binom{x+y-1}{x-1}^{-1}\ =\ \binom{x+y}{x}^{-1}
\end{align*}
for $x,y\in\N$, we find
\begin{align*}
P\wh{f}_{1}(x,y)\ &=\ \lambda\left(\frac{x}{x+y}\,\wh{f}_{1}(x-1,y)\,+\,\frac{y}{x+y}\,\wh{f}_{1}(x,y-1)\right)\\
&\hspace{.8cm}+\ \ovl{\lambda}\left(\frac{\wh{f}_{1}(x+1,y)\,+\,\wh{f}_{1}(x,y+1)}{2}\right)\\
&=\ \wh{f}_{1}(x,y)\,+\,\frac{1}{2}\,\binom{x+y}{x}^{-1}\,\mu^{x}\left(\lambda\,\frac{x+1}{x+y+1}\,+\,\ovl{\lambda}\,\frac{y+1}{x+y+1}\right)\\
&\hspace{.8cm}+\ \frac{1}{2}\,\binom{x+y}{x}^{-1}\,\mu^{y}\left(\ovl{\lambda}\,\frac{x+1}{x+y+1}\,+\,\lambda\,\frac{y+1}{x+y+1}\right),
\end{align*}
and from this, the assertion is easily derived.\qed
\end{proof}

\begin{lemma}\label{lem7}
The functions $\wh{f}_{0}$ and $\wh{f}(x,y):={\binom{x+y}{x}}^{-1}(2\mu)^{x+y}=2^{x+y}\wh{f}_{0}(x,y)$ are both $P$-subharmonic.
\end{lemma}

\begin{proof}
Regarding $\wh{f}_{0}$, we obtain, for $x,y\in\N$,
\begin{align*}
P\wh{f}_{0}(x,y)\ &=\ 2\ovl{\lambda}\,\wh{f}_{0}(x,y)
\,+\,\frac{\lambda}{2}\left(1+\frac{1}{x+y+1}\right)\wh{f}_{0}(x,y)\\
&\ge\ \ovl{\lambda}\left(2+\frac{\mu}{2}\right)\wh{f}_{0}(x,y)\ \ge\ \wh{f}_{0}(x,y)
\end{align*}
and analogously
\begin{equation*}
P\wh{f}(x,y)\ =\ \ovl{\lambda}\,\wh{f}(x,y)
\,+\,\lambda\left(1+\frac{1}{x+y+1}\right)\wh{f}(x,y)\ \ge\ \wh{f}(x,y).\qed
\end{equation*}
\end{proof}

\begin{lemma}\label{lem8}
The function $\wh{h}$ is $P$-subharmonic.
\end{lemma}

\begin{proof}
Noting that $\wh{h}=\wh{f}_{1}-\wh{f}_{0}$ and $\frac{\lambda}{2}\wh{f}_{1}(x,y)\ge\wh{f}_{0}(x,y)$ for $x,y\in\N$, the previous calculations provide us with
\begin{align*}
P\wh{h}(x,y)\ &=\ P\wh{f}_{1}(x,y)\,-\,P\wh{f}_{0}(x,y)\\
&\ge\ \wh{f}_{1}(x,y)\,+\,\frac{x+y+2}{x+y+1}\,\wh{f}_{0}(x,y)\\
&\qquad-\ \left(\frac{2}{1+\mu}\,+\,\frac{\mu}{2(1+\mu)}\,\frac{x+y+2}{x+y+1}\right)\wh{f}_{0}(x,y)\\
&\ge\ \wh{h}(x,y)\,+\,\frac{x+y+2}{x+y+1}\,\wh{f}_{0}(x,y)\\
&\qquad-\ \left(\frac{1-\mu}{1+\mu}\,+\,\frac{\mu}{2(1+\mu)}\,\frac{x+y+2}{x+y+1}\right)\wh{f}_{0}(x,y)\\
&\ge\ \wh{h}(x,y)\,+\,\left(1-\frac{2-\mu}{2(1+\mu)}\right)\frac{x+y+2}{x+y+1}\,\wh{f}_{0}(x,y)\\
&=\ \wh{h}(x,y)\,+\,\frac{\mu}{2(1+\mu)}\,\frac{x+y+2}{x+y+1}\,\wh{f}_{0}(x,y)\\
&\ge\ \wh{h}(x,y)
\end{align*}
for all $x,y\in\N$ (including the case $x=y$).\qed
\end{proof}

\subsection{Proofs of Thm.~\ref{thm:annealed inequality} and of Prop.~\ref{prop:standard}}

\begin{proof}[of Thm.~\ref{thm:annealed inequality}] Since $(X_{n},Y_{n})_{n\ge 0}$ exhibits standard behavior (see \eqref{eq:standard behavior} in Prop.~\ref{prop:standard}) and, using Lem.~\ref{lem4} and \ref{lem8}, $h,\wh{h}$ are obviously functions satisfying the conditions of Lem.~\ref{lem1}, we directly infer $\wh{h}(x,y)\le q_{x,y}\le h(x,y)$ for all $x,y\in\N_{0}$. As for $f_{0}$, it does not meet the boundary conditions stated in Lem.~\ref{lem1}, yet
\begin{align*}
f_{0}(x,y)\ &=\ \lim_{n\to\infty}\Erw_{x,y}\mu^{X_{\tau\wedge n}+Y_{\tau\wedge n}}\\
&=\ \int_{\{X_{\tau}=0\}}\mu^{Y_{\tau}}\ d\Prob_{x,y}\ +\ \int_{\{Y_{\tau}=0\}}\mu^{X_{\tau}}\ d\Prob_{x,y}\\
&\le\ \Prob_{x,y}(X_{\tau}=0)\,+\,\Prob_{x,y}(Y_{\tau}=0)\\
&=\ q_{x,y}
\end{align*}
for all $x,y\in\N_{0}$ as asserted. Turning to the proof of \eqref{eq:improved bounds kappa}, it suffices to note that
\begin{equation*}
\lim_{x\to\infty}f_{0}(x,x)^{1/x}\ =\ \mu^{2},\quad\lim_{x\to\infty}h(x,x)^{1/x}\ =\ \mu
\end{equation*}
and
\begin{equation*}
\lim_{x\to\infty}\wh{h}(x,x)^{1/x}\ =\ \lim_{x\to\infty}\binom{2x}{x}^{-1/x}h(x,x)^{1/x}\ =\ \frac{\mu}{4},
\end{equation*}
having used Stirling's formula for the last assertion.\qed
\end{proof}

\begin{proof}[of Prop.~\ref{prop:standard}] Recall from Subsec.~\ref{subsec:2-type splitting} that, on $\{\tau=\infty\}$, $X_{n}+Y_{n}$ can be viewed as the total population size of a $2$-type Bellman-Harris process at its $n${th} jump epoch. The extinction-explosion principle for such branching processes (see Thm.~(6.5.2) in \cite{Jagers:75}) implies that $X_{n}+Y_{n}\to\infty$ $\Prob_{x,y}$-a.s.\ on $\{\tau=\infty\}$ for all $x,y\in\N$. Now use that $(h(X_{n},Y_{n}))_{n\ge 0}$ forms a bounded supermartingale under any $\Prob_{x,y}$ and thus converges a.s. Consider the event 
$$ E\ :=\ \bigl\{\tau=\infty,\,\lim_{n\to\infty}X_{n}=\infty,\,\limsup_{n\to\infty}Y_{n}<\infty\bigr\} $$
and write
$$ h(X_{n},Y_{n})\ =\ 1-(1-\mu^{X_{n}})(1-\mu^{Y_{n}}). $$
Then we see that the integer-valued $Y_{n}$ must a.s.\ eventually stay constant on $E$. But this is impossible because at each birth epoch $Y_{n}$ changes by $+1$ with probability $\frac{1}{2}$. So $\Prob_{x,y}(E)=0$. Similarly, we find 
$$ \Prob\bigl(\tau=\infty,\,\lim_{n\to\infty}Y_{n}=\infty,\,\limsup_{n\to\infty}X_{n}<\infty\bigr)\,=\,0 $$
and so $X_{n}\wedge Y_{n}\to\infty$ a.s.\ on $\{\tau=\infty\}$ as claimed.\qed
\end{proof}

\subsection{On the construction and interpretation of the given harmonic functions}
\label{subsec:construction_sub(super)harmonic}

A central role in the present work is played by the use of harmonic and sub(super)harmonic functions. Indeed, the various functions that are provided in Subsec.~\ref{subsec:expression_harmonic} are all of this kind and used to get bounds for the lower and upper exponential decay $\kappa_{*}$ and $\kappa^{*}$ of the extinction probability,  respectively, which forms a crucial tool to establish our main results. A similar approach was used by \cite{AlsRoe:96,AlsRoe:02} to find bounds for the extinction probability ratio when comparing a bisexual Galton-Watson branching process with promiscuous mating with its asexual counterpart, see also the article by \cite{DaleyHullTaylor:86}. A key point in our analysis is that the functions $f_0$, $f_1$, $h$, etc., and their counterparts $\wh{f_0}$, $\wh{f_1}$, $\wh{h}$, etc., are explicit and rather simple. In this subsection, we would like to discuss some aspects of the construction of these functions.

\vspace{.1cm}
Finding exact expressions for harmonic functions reveals intrinsic properties of the model at hand. For instance, harmonic functions can be used to define martingales, which in turn yield information on the pathwise behavior of the random walks. Of related interest in this context is the constructive theory of Lyapunov functions for nonhomogeneous random walks, see \cite{FayMalMen:95,MenPoWa:17}.

\vspace{.1cm}
Our (sub,super)harmonic functions can be interpreted both combinatorially and probabilistically. From a combinatorial viewpoint, they are constructed from classical binomial coefficients $\binom{x+y}{x}$ and power functions $\mu^x$, $\mu^{x+y}$, etc. On the probabilistic side, they can be interpreted as absorption probabilities for related models (which later will appear in some coupling arguments):
\begin{itemize}\itemsep4pt
\item First, $f_0(x,y)$ in \eqref{eq1:various_functions_important} is the probability that simple random walk on the nonnegative integers with probabilities $\lambda$ and $\ovl{\lambda}$ for making a jump to the left and right, respectively, is eventually absorbed at the origin when starting from $x+y$.  

\item The quantity $h(x,y)$ in \eqref{eq2:various_functions_important} is the probability that a homogeneous random walk in the positive quadrant $\N^{2}$ which jumps to the four nearest neighbors $\rightarrow$, $\uparrow$, $\leftarrow$ and $\downarrow$ with respective probabilities\ $\frac{\lambda}{2}$, $\frac{\lambda}{2}$, $\frac{\overline{\lambda}}{2}$ and $\frac{\overline{\lambda}}{2}$ (see right panel of Fig.~\ref{fig:homogeneous models}) is eventually absorbed at the boundary when starting from $(x,y)$, see \cite{KurkovaRaschel:11,BilliardTran:12}.

\item Finally, the binomial coefficient $\binom{x+y}{x}^{-1}$ has a simple interpretation, since the probability of the path
\begin{equation*}
     (x,y) \to (x,y-1) \to (x,y-2) \to \cdots \to (x,1) \to (x,0)
\end{equation*}
is exactly $(\frac{\mu}{1+\mu})^{y}\binom{x+y}{x}^{-1}$. This path is the shortest one for the random walk to get absorbed (if $x<y$). Notice further that this binomial coefficient does respect the symmetry of the model, as obviously $\binom{x+y}{x}^{-1}=\binom{x+y}{y}^{-1}$.
\end{itemize}

\section{Quenched harmonic analysis and proof of Thm.~\ref{thm:q_x,y asymptotic}}\label{sec:proof theorem 2}

\subsection{Harmonic analysis when given the birth-death environment}\label{subsec:quenched analysis}
Adopting the framework of Subsec.~\ref{subsec:MCRE}, we now turn to an analysis of the model given the iid random environment $\bfe=(e_{1},e_{2},\ldots)$, i.e., under the probability measures $\Prob_{x,y}^{\,\bfe}$, where the value of $e_{n}$ marks whether the $n${th} jump is a birth ($+1$) or a death ($-1$). As we explained in that subsection, the sequence $(X_{n},Y_{n})_{n\ge 0}$ then becomes a Markov chain with iid random transition probabilities $p_{(x,y),(x\pm1,y\pm 1)}(e_{n})$ which are displayed in \eqref{eq:random transition probabilities} for $x,y\in\N$ and $p_{(x,0),(x,0)}(e_{n})=p_{(0,y),(0,y)}(e_{n})=1$ for $x,y\in\N_{0}^{2}$. 

\vspace{.1cm}
Put $T_{0}:=0$ and let $T_{1}+1,T_{2}+2,\ldots$ denote the successive epochs when $e_{n}=-1$, thus
\begin{align*}
T_{n}+n\ :=\ \inf\{k>T_{n-1}+n-1:e_{k}=-1\}
\end{align*}
for each $n\ge 1$. Notice that all $T_{n}$ are measurable with respect to $\bfe$ and thus constants under any $\Prob_{x,y}^{\,\bfe}$. Furthermore, $(T_{n})_{n\ge 0}$ has iid increments $\chi_{1},\chi_{2},\ldots$ under any $\Prob_{x,y}$ with a geometric distribution on $\N_{0}$, more precisely
$$ \Prob(\chi_{1}=n)\ =\ \lambda\,\ovl{\lambda}^{n} $$
for $n\in\N_{0}$. Its generating function $\vph(\theta)=\Erw\,\theta^{\chi_{1}}$ equals
$$ \vph(\theta)\ =\ \frac{\lambda}{1-\ovl{\lambda}\theta} $$
for $\theta<\ovl{\lambda}^{-1}$, giving in particular $\vph((2\ovl{\lambda})^{-1})=2\lambda$. As a direct consequence, to be used later on, we note that:

\begin{lemma}\label{lem:product martingale}
Under each $\Prob_{x,y}$, the sequence
$$ \left(\frac{1}{(2\lambda)^{n}(2\ovl{\lambda})^{T_{n}}}\right)_{n\ge 0} $$
forms a product martingale with mean $1$ and almost sure limit $0$.
\end{lemma}

Since absorption at the axes can clearly occur only at the $T_{n}+n$, we will study hereafter the behavior of the subsequence
$$ (\wh{X}_{n},\wh{Y}_{n})_{n\ge 0}\,:=\,(X_{T_{n}+n},Y_{T_{n}+n})_{n\ge 0} $$
under $\Prob_{x,y}^{\,\bfe}$. Recall from Subsec.~\ref{subsec:MCRE} the definition of the transition operators $P_1$ and $P_{-1}$.

\begin{lemma}\label{lem:E_x,y^e f}
Put $\wh{P}_{n}:=P_{1}^{\chi_{n}}P_{-1}$. Then $(\wh{X}_{n},\wh{Y}_{n})_{n\ge 0}$ is a nonhomogeneous Markov chain under $\Prob_{x,y}^{\,\bfe}$ with transition operators $\wh{P}_{1},\wh{P}_{2},\ldots$, thus
\begin{align}\label{eq:E_x,y^e f}
\wh{P}_{1}\cdots\wh{P}_{n}f(x,y)\ =\ \Erw_{x,y}^{\,\bfe}f(\wh{X}_{n},\wh{Y}_{n})
\end{align}
for any $x,y,n\in\N$ and any nonnegative function $f$ on $\N_{0}^{2}$.
\end{lemma}

\begin{proof}
Recalling that $\chi_{1},\chi_{2},\ldots$ and thus the $T_{n}$ are measurable with respect to $\bfe$ and that $(X_{n},Y_{n})_{n\ge 0}$ is Markovian under $\Prob_{x,y}^{\bfe}$, we obtain
\begin{align*}
\Erw_{x,y}^{\bfe}f(\wh{X}_{n},\wh{Y}_{n})\,&=\,\Erw_{x,y}^{\bfe}\Erw_{x,y}^{\bfe}(f(\wh{X}_{n},\wh{Y}_{n})|\wh{X}_{n-1},\wh{Y}_{n-1})\,=\,\Erw_{x,y}^{\bfe}\wh{P}_{n}f(\wh{X}_{n-1},\wh{Y}_{n-1})
\end{align*}
and then \eqref{eq:E_x,y^e f} upon successive conditioning.\qed
\end{proof}

\vspace{.1cm}
The following lemma provides the crucial information about the spectral properties of $P_{1}$ and $P_{-1}$ with respect to the functions $f_{0}$ and $f_{1}$ introduced in \eqref{eq1:various_functions_important}. Let us put
\begin{equation}
\label{eq:def_gamma}
     \gamma\ :=\ \frac{1-\mu}{1+\mu}\ =\ \ovl{\lambda}-\lambda\ =\ 1-2\lambda\ \in(0,1).
\end{equation}

\begin{lemma}\label{lem:quenched spectral lemma}
For all $x,y\in\N$, the following assertions hold true:
\begin{align}
P_{1}f_{0}(x,y)\ &=\ \mu f_{0}(x,y)\quad\text{and}\quad P_{-1}f_{0}(x,y)\ =\ \mu^{-1}f_{0}(x,y),\label{eq:P_1f_0}\\
P_{1}f_{1}(x,y)\ &=\ \frac{1+\mu}{2}\,f_{1}(x,y)\ =\ \frac{1}{2\ovl{\lambda}}\,f_{1}(x,y),\label{eq:P_1f_1}\\
\begin{split}
P_{-1}f_{1}(x,y)\ &=\ \frac{1-\delta_{1}(x,y)}{2}(\mu^{x-1}+\mu^{x}+\mu^{y}+\mu^{y-1})\\
&=\ \frac{1-\delta_{1}(x,y)}{2\lambda}\,f_{1}(x,y),
\label{eq1:P_-1f_1}
\end{split}
\shortintertext{where}
\delta_{1}(x,y)\ &=\ \gamma\,\frac{1-\mu^{|x-y|}}{1+\mu^{|x-y|}}\,\frac{|x-y|}{x+y}.\nonumber
\end{align}
\end{lemma}

\begin{proof}
We will only prove \eqref{eq1:P_-1f_1}, for all other identities are readily checked. Let us start by noting that, if $y=x+m$ for some $m\in\N_{0}$, then
$$ \frac{1}{2}-\frac{x}{x+y}\ =\ \frac{y}{x+y}-\frac{1}{2}\ =\ \frac{m}{2(x+y)} $$
and therefore
\begin{multline*}
\left(\frac{1}{2}-\frac{x}{x+y}\right)(\mu^{x-1}+\mu^{y})\,+\,\left(\frac{1}{2}-\frac{y}{x+y}\right)(\mu^{x}+\mu^{y-1})\\
=\ \frac{m}{2(x+y)}\,\mu^{x-1}(1-\mu^{m})(1-\mu).
\end{multline*}
It follows that
\begin{align*}
P_{-1}f_{1}(x,y)\ &=\ \frac{x}{x+y}\,(\mu^{x-1}+\mu^{y})\,+\,\frac{y}{x+y}\,(\mu^{x}+\mu^{y-1})\\
&=\ \frac{1}{2}\,\mu^{x-1}(1+\mu^{m})(1+\mu)\,+\,\frac{m}{2(x+y)}\,\mu^{x-1}(1-\mu^{m})(1-\mu),
\end{align*}
and with this at hand, it remains to assess for \eqref{eq1:P_-1f_1} that the last term in the previous line equals
\begin{align*}
\frac{\delta_{1}(x,y)}{2}(\mu^{x-1}+\mu^{x}+\mu^{y}+\mu^{y-1})\ =\ \frac{\delta_{1}(x,y)}{2}\,\mu^{x-1}(1+\mu^{m})(1+\mu)
\end{align*}
and that $m=y-x=|y-x|$.\qed
\end{proof}

Note that $\delta_{1}(x,x)=0$ and
\begin{align}\label{eq:bounds for delta}
\gamma^{2}\,\frac{|x-y|}{x+y}\ \le\ \delta_{1}(x,y)\ \le\ \gamma\,\frac{|x-y|}{x+y}
\end{align}
for all $x,y\in\N$.

\begin{lemma}\label{lem:iterations Phat for f_1}
For all $n\in\N_{0}$ and $x,y\in\N$ such that $(x\wedge y)\wedge|y-x|\ge n$,
\begin{equation}
\frac{1-a_{n}(x,y)}{2\lambda(2\ovl{\lambda})^{n}}\,f_{1}(x,y)\ \le\ P_{1}^{n}P_{-1}f_{1}(x,y)\ \le\ \frac{1-\gamma\,a_{n}(x,y)}{2\lambda(2\ovl{\lambda})^{n}}\,f_{1}(x,y),\label{eq:P_1^n P_-1f_1}
\end{equation}
where
\begin{equation*}
a_{n}(x,y)\ :=\ \gamma\,\frac{|x-y|}{x+y+n}+\gamma^{2}\,\frac{n}{x+y+n}.
\end{equation*}
\end{lemma}

\begin{proof}
We use induction over $n\in\N_{0}$. For $n=0$, the result follows from \eqref{eq1:P_-1f_1} and \eqref{eq:bounds for delta}. Put $b_{n}^{(i)}(x,y):=1-\gamma^{i}\,a_{n}(x,y)$ for $i\in\N_{0}$. Assuming \eqref{eq:P_1^n P_-1f_1} be true for $n$, we obtain for $x,y\in\N$, w.l.o.g.~$x\le y$, such that $x\wedge(y-x)\ge n+1$
\begin{align*}
&2\lambda(2\ovl{\lambda})^{n}\,P_{1}^{n+1}P_{-1}f_{1}(x,y)\ \ge\ P_{1}\!\left[b_{n}^{(0)}(x,y)f_{1}(x,y)\right]\\
&\hspace{1cm}=\ \frac{1}{2}\left(b_{n}^{(0)}(x+1,y)(\mu^{x+1}+\mu^{y})+b_{n}^{(0)}(x,y+1)(\mu^{x}+\mu^{y+1})\right)\\
&\hspace{1cm}=\ \frac{1}{2}\left(b_{n+1}^{(0)}(x,y)+\frac{\gamma^{2}}{x+y+n+1}\right)\left(\mu^{x+1}+\mu^{y}+\mu^{x}+\mu^{y+1}\right)\\
&\hspace{1.5cm}+\ \frac{1}{2}\left(\frac{\gamma}{x+y+n+1}\right)\big(\mu^{x+1}+\mu^{y}-\mu^{x}-\mu^{y+1}\big)\\
&\hspace{1cm}=\ \frac{1+\mu}{2}\left(b_{n+1}^{(0)}(x,y)+\frac{\gamma^{2}}{x+y+n+1}\right)\left(\mu^{x}+\mu^{y}\right)\\
&\hspace{1.5cm}-\ \frac{1}{2}\left(\frac{\gamma}{x+y+n+1}\right)\big(\mu^{x}+\mu^{y}\big)(1-\mu)\\
&\hspace{1cm}=\ \frac{1}{2\ovl{\lambda}}\,b_{n+1}^{(0)}(x,y)\,f_{1}(x,y)
\end{align*}
and in exactly the same way
\begin{equation*}
2\lambda(2\ovl{\lambda})^{n}\,P_{1}^{n+1}P_{-1}f_{1}(x,y)\ \le\ P_{1}\!\left[b_{n}^{(1)}(x,y)\,f_{1}(x,y)\right]\ =\ \frac{1}{2\ovl{\lambda}}\,b_{n+1}^{(1)}(x,y)\,f_{1}(x,y)
\end{equation*}
which proves the assertion.\qed
\end{proof}

\subsection{Proof of Thm.~\ref{thm:q_x,y asymptotic}}\label{subsec:proof theorem 2}

Given a sequence $(Z_{n})_{n\ge 0}$ of random variables, we stipulate for our convenience that its extension to the time domain $[0,\infty)$ is defined by $Z_{t}:=Z_{\lfloor t\rfloor}$ for $t\ge 0$. 
For $c\ge 0$, let
$$ \tau_{c}\ :=\ \inf\{n\ge 0:X_{n}\wedge Y_{n}\le c\}, $$
thus $\tau_{c}=\tau_{\lfloor c\rfloor}$ and $\tau=\tau_{0}$.

\vspace{.1cm}
The proof of Thm.~\ref{thm:q_x,y asymptotic} is furnished by a number of lemmata, but let us sketch its main arguments first. In order to hit one of the axes, the random walk $(X_{n},Y_{n})_{n\ge 0}$, when starting at $(x,x)$, must clearly first hit one of the halflines
$$ \{(y,\lfloor\beta x\rfloor):\N\ni y\ge\lfloor\beta x\rfloor\}\quad\text{or}\quad\{(\lfloor\beta x\rfloor,y):\N\ni y\ge\lfloor\beta x\rfloor\} $$
for any $\beta\in (0,1)$ (see Fig.~\ref{fig:subplane}), and the probability for this to happen, that is for $\tau_{\beta x}$ to be finite, can easily be bounded by $2\mu^{(1-\beta)x}$, see Lem.~\ref{lem:tau_beta x finite}. On the other hand, it can further be shown for sufficiently large $x$ that $|X_{\tau_{\beta x}}-Y_{\tau_{\beta x}}|$ is not too small with very high probability, namely larger than $\alpha x$ for some $\alpha>0$, see Lem.~\ref{lem:discrepancy}. With the help of the strong Markov property and the obvious fact that $(\tau_{\beta x},X_{\tau_{\beta x}})\eqdist (\tau_{\beta x},Y_{\tau_{\beta x}})$ under $\Prob_{x,x}$, it then follows that 
\begin{align*}
q_{x,x}\ &\le\ 2\int_{\{\tau_{\beta x}<\infty,X_{\tau_{\beta x}}=\lfloor\beta x\rfloor,Y_{\tau_{\beta x}}-X_{\tau_{\beta x}}>\alpha x\}}q_{\lfloor\beta x\rfloor,Y_{\tau_{\beta x}}}\ d\Prob_{x,x}\ +\ r(x)\\
&\le\ \Prob_{x,x}(\tau_{\beta x}<\infty)\,\sup_{y\ge (1+\alpha)\lfloor\beta x\rfloor}q_{\lfloor\beta x\rfloor,y}\ +\ r(x)\\
&\le\ 2\mu^{(1-\beta)x}\sup_{y\ge (1+\alpha)\lfloor\beta x\rfloor}q_{\lfloor\beta x\rfloor,y}\ +\ r(x)
\end{align*}
where $r(x)$ is a remainder of order $o(\mu^{(1+\eps)x})$ for some $\eps>0$. The proof of the theorem is finally completed by showing that, for some $\nu<\mu$, 
$$ \sup_{y\ge (1+\alpha)x}q_{x,y}\ =\ o(\nu^{x}) $$ 
as $x\to\infty$, see Lem.~\ref{lem:q(x,(1+alpha)x) behavior}. This is actually accomplished by choosing $\beta$ and then $\alpha$ in a appropriate manner.

\begin{lemma}\label{lem:tau_beta x finite}
For all $x,y\in\N$ and $c\in [0,x\wedge y)$,
\begin{equation*}
\Prob_{x,y}(\tau_{c}<\infty)\ \le\ f_{1}(x-c,y-c)\ \le\ 2\,\mu^{(x-c)\wedge (y-c)}.
\end{equation*}
In particular
\begin{equation}
\Prob_{x,y}(\tau_{\beta x}<\infty)\ \le\ 2\mu^{(1-\beta)x}\label{eq:tau_beta x finite}
\end{equation}
if $x\le y$ and $\beta\in (0,1)$.
\end{lemma}

\begin{proof}
Let $c$ be an integer. By Lem.~\ref{lem5}, $(h(X_{\tau_{c}\wedge n},Y_{\tau_{c}\wedge n}))_{n\ge 0}$ forms a bounded nonnegative supermartingale, and it converges $\Prob_{x,y}$-a.s.\ to $$h(c,Y_{\tau_{c}})\1_{\{\tau_{c}<\infty,X_{\tau_{c}}=c\}}+h(X_{\tau_{c}},c)\1_{\{\tau_{c}<\infty,Y_{\tau_{c}}=c\}}\ge\mu^{c}\,\1_{\{\tau_{c}<\infty\}}$$ because, by Prop.~\ref{prop:standard}, $X_{n}\wedge Y_{n}\to\infty$ a.s.\ on $\{\tau_{c}=\infty\}\subset\{\tau=\infty\}$. Consequently,
\begin{gather*}
\mu^{c}\,\Prob_{x,y}(\tau_{c}<\infty)\ \le\ \lim_{n\to\infty}\Erw_{x,y}h(X_{\tau_{c}\wedge n},Y_{\tau_{c}\wedge n})\ \le\ h(x,y)
\shortintertext{and therefore}
\Prob_{x,y}(\tau_{c}<\infty)\ \le\ \mu^{-c}\,h(x,y)\ \le\ f_{1}(x-c,y-c)\ \le\ 2\,\mu^{(x-c)\wedge (y-c)}
\end{gather*}
for all $x,y\in\N$ and $c\in [0,x\wedge y)$, as claimed.\qed
\end{proof}

\begin{lemma}\label{lem:discrepancy}
Given any $\beta\in (0,1)$ and $\alpha>0$,
\begin{equation*}
\Prob_{x,x}(|X_{\tau_{\beta x}}-Y_{\tau_{\beta x}}|\le\alpha x)\ \le\ \mu^{(2-2\beta-\alpha)x}
\end{equation*}
for all $x\in\N$.
\end{lemma}

\begin{proof}
The Optional Sampling Theorem provides us with
\begin{equation*}
\mu^{2x}\ =\ \Erw_{x,x}\mu^{X_{\tau_{\beta x}}+Y_{\tau_{\beta x}}}\ \ge\ \mu^{2\beta x}\,\Erw_{x,x}\mu^{X_{\tau_{\beta x}}\vee Y_{\tau_{\beta x}}-X_{\tau_{\beta x}}\wedge Y_{\tau_{\beta x}}},
\end{equation*}
thus
\begin{align*}
\mu^{2(1-\beta)x}\ &\ge\ \Erw_{x,x}\mu^{X_{\tau_{\beta x}}\vee Y_{\tau_{\beta x}}-X_{\tau_{\beta x}}\wedge Y_{\tau_{\beta x}}}\\ 
&\ge\ \mu^{\alpha x}\,\Prob_{x,x}(X_{\tau_{\beta x}}\vee Y_{\tau_{\beta x}}-X_{\tau_{\beta x}}\wedge Y_{\tau_{\beta x}}\le\alpha x)\\
&=\ \mu^{\alpha x}\,\Prob_{x,x}(|X_{\tau_{\beta x}}-Y_{\tau_{\beta x}}|\le\alpha x)
\end{align*}
which immediately implies the assertion.\qed
\end{proof}

Recall that $T_{n}+n$ denotes the epoch at which the $n$th death (downward step) occurs, so that $T_{n}$ provides the number of births (upward steps) until then.
Since any birth is equally likely to be of phenotype $\sfA$ (upward jump in the $x$-coordinate) and $\sfB$ (upward jump in the $y$-coordinate), the total number of $\sfA$-type births until $T_{n}+n$, say $S_{n}$, has a binomial distribution with parameters $T_{n}$ and $\frac{1}{2}$ under any $\Prob_{x,y}^{\,\bfe}$, and the increments $S_{k}-S_{k-1}$ are independent and binomial with parameters $\chi_{k}$ and $\frac{1}{2}$. Notice also that
\begin{align}
\begin{split}\label{eq:lower bound Yhat minus Xhat}
\wh{Y}_{n}-\wh{X}_{n}\ &=\ Y_{T_{n}+n}-X_{T_{n}+n}\\
&\ge\ y-x-n+T_{n}-2S_{n}\ \ge\ y-x-n-T_{n}
\end{split}
\end{align}
for all $n\in\N$, a fact to be used in the proof of the subsequent lemma.

\begin{lemma}\label{lem:prob E_beta,xi}
For any $\beta>0$, there exists $\xi>\mu^{-1}$ large enough such that
\begin{gather}
\Prob\left(T_{\beta x}>\xi\beta x\right)\ =\ o(\mu^{(1+2\beta)x})\label{eq:LDP neg binomial}
\shortintertext{and, for any $\alpha>(\xi+2)\beta$,}
\sup_{y\ge (1+\alpha)x}\Prob_{x,y}\left(\min_{1\le n\le\beta x}(\wh{Y}_{n}-\wh{X}_{n})\le\beta x\right)\ =\ o(\mu^{(1+\beta)x})\label{eq:E_beta,xi complement}
\end{gather}
as $x\to\infty$.
\end{lemma}

\begin{proof}
Fix $\beta>0$. Since, for each $n\in\N$, $T_{n}$ is the sum of $n$ iid geometric random variables with parameter $\lambda$ (thus mean $\ovl{\lambda}/\lambda=\mu^{-1}$) under any $\Prob_{x,y}$, Cram\'er's theorem implies that, for sufficiently large $\xi>\mu^{-1}$ and $n\to\infty$,
$$ \Prob\left(\frac{T_{n}}{n}>\xi\right)\ =\ o(\mu^{(2+1/\beta)n}) $$
and thus \eqref{eq:LDP neg binomial} holds true as $x\to\infty$. Now pick an arbitrary $\alpha>(\xi+2)\beta$. Using \eqref{eq:lower bound Yhat minus Xhat}, we then obtain for $1\le n\le\beta x$ and $y\ge (1+\alpha)x$
\begin{align*}
\Prob_{x,y}\left(\wh{Y}_{n}-\wh{X}_{n}\le\beta x\right)\ &\le\ \Prob\left(y-x-n-T_{n}\le\beta x\right)\\
&\le\ \Prob\left(T_{n}\ge (\alpha-2\beta)x\right)\\
&\le\ \Prob\left(T_{\beta x}\ge \xi\beta x\right)
\end{align*}
and thereupon with the help of \eqref{eq:LDP neg binomial}
\begin{align*}
\Prob_{x,y}\left(\min_{1\le n\le\beta x}(\wh{Y}_{n}-\wh{X}_{n})\le\beta x\right)\ &\le\ \sum_{1\le n\le\beta x}\Prob_{x,y}\left(\wh{Y}_{n}-\wh{X}_{n}\le\beta x\right)\\
&\le\ \beta x\,\Prob\left(T_{\beta x}\ge \xi\beta x\right)\ =\ o(\mu^{(1+\beta)x})
\end{align*}
as $x\to\infty$.\qed
\end{proof}

\begin{lemma}\label{lem:the crucial one}
Let $\beta\in (0,1)$, $\xi>0$, $\alpha>(\xi+3)\beta$, $y=(1+\alpha)x$, $\gamma$ defined in \eqref{eq:def_gamma}, and put $n(x):=\lfloor\beta x\rfloor+1$. Then there exists $\theta=\theta(\alpha,\beta,\gamma,\xi)\in (0,1)$ such that, for all sufficiently large $x$,
\begin{gather*}
\sup_{1\le n\le n(x)}\frac{(2\lambda)^{n}(2\ovl{\lambda})^{T_{n}}}{\theta^{n}}\,\wh{P}_{1}\cdots\wh{P}_{n}f_{1}(x,y)\ \le\ f_{1}(x,y)\quad\Prob_{x,y}^{\,\bfe}\text{-a.s.}
\intertext{on the event}
E_{\beta,\xi}\ :=\ \left\{\min_{1\le n\le n(x)}|\wh{X}_{n}-\wh{Y}_{n}|>\beta x,\,T_{n(x)}\le \xi n(x)\right\},
\intertext{in particular}
\wh{P}_{1}\cdots\wh{P}_{n(x)}f_{1}(x,y)\ \le\ \frac{\theta^{\beta x}}{(2\lambda)^{n(x)}(2\ovl{\lambda})^{T_{n(x)}}}\,f_{1}(x,y)\quad\Prob_{x,y}^{\,\bfe}\text{-a.s.}
\end{gather*}
\end{lemma}

\begin{proof}
Recall from Lem.~\ref{lem:iterations Phat for f_1} the definition of $a_{n}(x,y)$ and observe that $\Prob_{x,y}^{\,\bfe}$-a.s.\ on $E_{\beta,\xi}$
\begin{align*}
a_{\chi_{n}}(\wh{X}_{n-1},\wh{Y}_{n-1})\ &\ge\ \gamma\,\frac{|\wh{X}_{n-1}-\wh{Y}_{n-1}|}{\wh{X}_{n-1}+\wh{Y}_{n-1}+\chi_{n}}\\
&=\ \gamma\,\frac{|\wh{X}_{n-1}-\wh{Y}_{n-1}|}{x+y+T_{n}+n-1}\ \ge\ \frac{\gamma\beta}{2+\alpha+\beta(1+\xi)}
\end{align*}
for all $n=1,\ldots,n(x)-1$. With this at hand, we use Lem.~\ref{lem:E_x,y^e f} and \ref{lem:iterations Phat for f_1} to infer for such $n$
\begin{align*}
&\wh{P}_{1}\cdots\wh{P}_{n}f_{1}(x,y)\ =\ \Erw_{x,y}^{\bfe}\wh{P}_{n}(\wh{X}_{n-1},\wh{Y}_{n-1})\\
&\quad\le\ \Erw_{x,y}^{\bfe}\left[\frac{1-\gamma\,a_{\chi_{n}}(\wh{X}_{n-1},\wh{Y}_{n-1})}{2\lambda(2\ovl{\lambda})^{\chi_{n}}}\,f_{1}(\wh{X}_{n-1},\wh{Y}_{n-1})\right]\\
&\quad\le\ \frac{1}{2\lambda(2\ovl{\lambda})^{\chi_{n}}}\,\left(1-\frac{\gamma^{2}\beta}{2+\alpha+\beta(1+\xi)}\right)\Erw_{x,y}^{\bfe}f_{1}(\wh{X}_{n-1},\wh{Y}_{n-1})\\
&\quad=\ \frac{1}{2\lambda(2\ovl{\lambda})^{\chi_{n}}}\,\left(1-\frac{\gamma^{2}\beta}{2+\alpha+\beta(1+\xi)}\right)\wh{P}_{1}\cdots\wh{P}_{n-1}f_{1}(x,y)\quad\Prob_{x,y}^{\,\bfe}\text{-a.s.}
\end{align*}
Upon setting
\begin{equation}\label{eq:def theta}
\theta\ :=\ 1-\frac{\gamma^{2}\beta}{2+\alpha+\beta(1+\xi)}
\end{equation}
and iteration, the assertions now easily follow.\qed
\end{proof}

\begin{lemma}\label{lem:q(x,(1+alpha)x) behavior}
Let $E_{\beta,\xi}$ be the set defined in Lem.~\ref{lem:the crucial one}. Given any $\beta>0$, let $\xi>\mu^{-1}$ be such that, by Lem.~\ref{lem:prob E_beta,xi}, $\Prob_{x,y}(E_{\beta,\xi}^{c})=o(x^{(1+\beta)x})$ as $x\to\infty$. Then
\begin{equation*}
\sup_{y\ge (1+\alpha)x}q_{x,y}\ \le\ \theta^{\beta x}\mu^{x}+o(\mu^{(1+\beta)x})
\end{equation*}
for $\theta$ defined in \eqref{eq:def theta},  $\alpha>(\xi+2)\beta$ and all $x\in\N$.
\end{lemma}

\begin{proof}
For $\alpha,\beta,\xi$ as claimed and $n(x)=\lfloor\beta x\rfloor+1$, we obtain with the help of the previous lemmata in combination with Lem.~\ref{lem:product martingale}
\begin{align*}
q_{x,y}\ &\le\ \Erw_{x,y}h(X_{T_{n(x)}+n(x)},Y_{T_{n(x)}+n(x)})\ =\ \Erw_{x,y}h(\wh{X}_{n(x)},\wh{Y}_{n(x)})\\ 
&\le\ \Erw_{x,y}f_{1}(X_{T_{n(x)}+n(x)},Y_{T_{n(x)}+n(x)})\ =\ \wh{P}_{1}\cdots\wh{P}_{n(x)}f_{1}(x,y)\\
&\le\ \Erw_{x,y}\left(\frac{1}{(2\lambda)^{n(x)}(2\ovl{\lambda})^{T_{n(x)}}}\right)\theta^{\beta x}\,f_{1}(x,y)\ +\ \Prob_{x,y}(E_{\beta,\xi}^{c})\\
&\le\ \theta^{\beta x}(\mu^{x}+\mu^{(1+\alpha)x})\ +\ o(\mu^{(1+\beta)x})
\end{align*}
for all $x\in\N$ and $y\ge (1+\alpha)x$.\qed
\end{proof}

We are now in position to prove Thm.~\ref{thm:q_x,y asymptotic}.

\begin{proof}[of Thm.~\ref{thm:q_x,y asymptotic}]
Fix $\alpha>0$ such that, by Lem.~\ref{lem:q(x,(1+alpha)x) behavior}, $q_{x,y}=o(\mu^{(1+\eps)x})$ for all $x\in\N$, $y\ge (1+\alpha)x$ and some $\eps>0$. Then pick $\beta\in (0,1)$ so small that $(2+\alpha)\beta<1-\eps$. Lem.~\ref{lem:discrepancy} then provides us with
\begin{equation*}
\Prob_{x,x}(|X_{\tau_{\beta x}}-Y_{\tau_{\beta x}}|\le\beta\alpha x)\ \le\ \mu^{(2-\beta(2+\alpha))x}\ =\ o(\mu^{(1+\eps)x}).
\end{equation*}
Note also that, by Lem.~\ref{lem:tau_beta x finite},
$$ \Prob_{x,y}(\tau_{\beta x}<\infty)\ \le\ 2\mu^{(1-\beta)x} $$
for all $x\in\N$ and $y\ge (1+\alpha)x$. By combining these facts and using the strong Markov property, we now obtain
\begin{align*}
q_{x,x}\ &=\ \int_{\{\tau_{\beta x}<\infty\}}q_{X_{\tau_{\beta x}},Y_{\tau_{\beta x}}}\ d\Prob_{x,x}\\
&\le\ \Prob_{x,x}(|X_{\tau_{\beta x}}-Y_{\tau_{\beta x}}|\le\beta\alpha x)\\
&\hspace{2.5cm}+\ \int_{\{|X_{\tau_{\beta x}}-Y_{\tau_{\beta x}}|>\beta\alpha x,\tau_{\beta x}<\infty\}}q_{X_{\tau_{\beta x}},Y_{\tau_{\beta x}}}\ d\Prob_{x,x}\\
&\le\ \Prob_{x,x}(|X_{\tau_{\beta x}}-Y_{\tau_{\beta x}}|\le\beta\alpha x)\ +\ q_{\beta x,(1+\alpha)\beta x}\,\Prob_{x,x}(\tau_{\beta x}<\infty)\\
&\le\ o(\mu^{(1+\eps)x})\ +\ o(\mu^{(1+\eps)\beta x})2\mu^{(1-\beta)x}\ =\ o(\mu^{(1+\eps\beta)x})
\end{align*}
and this proves our assertion.\qed
\end{proof}

\section{Proof of Thm.~\ref{thm:q(x,y) monotonic}}\label{sec:proof theorem 3}

Since $q_{x,y}=q_{y,x}$, it suffices to prove $q_{x,y}\ge q_{x+1,y}$ for all $x,y\in\N$. Let $(X_{n},Y_{n})_{n\ge 0}$ and $(X_{n}',Y_{n}')_{n\ge 0}$ be two coupled Markov chains on a common probability space $(\Omega,\fA,\Prob)$ with increments $(\zeta_{n},\chi_{n})$ and $(\zeta_{n}',\chi_{n}')$, respectively, and joint canonical filtration $(\cG_{n})_{n\ge 0}$ such that the following conditions hold:
\begin{enumerate}[label={\rm (C\arabic{*})},ref={\rm (C\arabic{*})},leftmargin=1.7\parindent]\itemsep2pt
     \item\label{it:C1}$(X_{0},Y_{0})=(x,y)$ and $(X_{0}',Y_{0}')=(x+1,y)$.
     \item\label{it:C2}Both chains have transition kernel $P$.
     \item\label{it:C3}The two ordinary random walks $$(S_{n})_{n\ge 0}:=(X_{n}+Y_{n})_{n\ge 0}\quad \text{and}\quad (S_{n}')_{n\ge 0}:=(X_{n}'+Y_{n}')_{n\ge 0}$$ have the same increments $\xi_{1},\xi_{2},\ldots$, but starting points $S_{0}=x+y$ and $S_{0}'=x+y+1$, thus $S_{n}'-S_{n}=1$ for all $n\in\N_{0}$.
     \item\label{it:C4}The conditional laws of $(\zeta_{n},\chi_{n})$ and $(\zeta_{n}',\chi_{n}')$ given $\xi_{n}$ and $\cG_{n-1}$ are specified as follows:
If $(X_{n-1},Y_{n-1})=(x,y),\,(X_{n-1}',Y_{n-1}')=(x',y')\in\{(x,y+1),(x+1,y)\}$ and $S_{n-1}=S_{n-1}'-1=x+y=:s$, then
\begin{align*}
&\Prob((\zeta_{n},\chi_{n})=(\zeta_{n}',\chi_{n}')=(1,0)|\xi_{n}=1,\cG_{n-1})=\frac{1}{2},\\
&\Prob((\zeta_{n},\chi_{n})=(\zeta_{n}',\chi_{n}')=(0,1)|\xi_{n}=1,\cG_{n-1})=\frac{1}{2},\\
&\Prob((\zeta_{n},\chi_{n})=(\zeta_{n}',\chi_{n}')=(-1,0)|\xi_{n}=-1,\cG_{n-1})=\frac{x}{s}\wedge\frac{x'}{s+1},\\
&\Prob((\zeta_{n},\chi_{n})=(\zeta_{n}',\chi_{n}')=(0,-1)|\xi_{n}=-1,\cG_{n-1})=\frac{y}{s}\wedge\frac{y'}{s+1},\\
&\Prob((\zeta_{n},\chi_{n})=(-1,0),(\zeta_{n}',\chi_{n}')=(0,-1)|\xi_{n}=-1,\cG_{n-1})\\
&\hspace{3.6cm}=\left(\frac{x}{s}-\frac{x'}{s+1}\right)^{+}\ =\ \frac{x}{s(s+1)}\,\1_{x=x'},\\
&\Prob((\zeta_{n},\chi_{n})=(0,-1),(\zeta_{n}',\chi_{n}')=(-1,0)|\xi_{n}=-1,\cG_{n-1})\\
&\hspace{3.6cm}=\left(\frac{y}{s}-\frac{y'}{s+1}\right)^{+}\ =\ \frac{y}{s(s+1)}\,\1_{y=y'}.
\end{align*}
\end{enumerate}
We claim that $(X_{n}',Y_{n}')$ equals either $(X_{n}+1,Y_{n})$ or $(X_{n},Y_{n}+1)$ for all $n\in\N_{0}$ and note that this is true for $n=0$ by \ref{it:C1}. Assuming it be true for all $k=0,\ldots,n-1$ (inductive hypothesis) and further $(X_{n-1},Y_{n-1})=(x,y),\,(X_{n-1}',Y_{n-1}')=(x',y')$, the claim must be checked for $(X_{n}',Y_{n}')$ only in the case when $\xi_{n}=-1$ and $(\zeta_{n},\chi_{n})$ and $(\zeta_{n}',\chi_{n}')$ take different values. But if this happens, then $x=x'$ leads to $(X_{n},Y_{n})=(x-1,y)$ and $(X_{n}',Y_{n}')=(x',y'-1)=(x,y)$, while $y=y'$ leads to $(X_{n},Y_{n})=(x,y-1)$ and $(X_{n}',Y_{n}')=(x'-1,y')=(x,y)$. This proves our claim. Finally, recalling that
$$ \tau\ =\ \inf\{n\ge 0:X_{n}=0\text{ or }Y_{n}=0\} $$
and defining $\tau'$ accordingly for the primed chain, we conclude $\tau\le\tau'$ and thus
$$ q_{x,y}\ =\ \Prob(\tau<\infty)\ \ge\ \Prob(\tau'<\infty)\ =\ q_{x+1,y}.\eqno\qed $$

\section{The absorption probabilities for \IBCOS\ and \BUTS}\label{sec:extinction IBCOS BUTS}

As pointed out in the Introduction, our model constitutes a hybrid of the two homogeneous models \IBCOS\ (independent branching with complete segregation) and \BUTS\ (branching with unbiased type selection) for which one-step transition probabilities of the associated random walk $(X_{n},Y_{n})_{n\ge 0}$ are shown in Fig.~\ref{fig:homogeneous models}. Homogeneity refers to the fact that these transition probabilities are the same regardless of whether a birth or a death of an individual has occurred. The next lemma can be checked very easily and is therefore stated without proof.

\begin{lemma}\label{lem:hom model}
Let $(X_{n},Y_{n})_{n\ge 0}$ be a random walk on $\N_{0}^{2}$ which is absorbed at the axes and has transition probabilities
\begin{align}
\begin{split}\label{eq:transition probabilities hom}
p_{(x,y),(x+1,y)}\ &=\ \ovl{\lambda}\,\phi(x,y),\quad p_{(x,y),(x,y+1)}\ =\ \ovl{\lambda}\,\ovl{\phi}(x,y),\\
p_{(x,y),(x-1,y)}\ &=\ \lambda\,\phi(x,y),\quad p_{(x,y),(x,y-1)}\ =\ \lambda\,\ovl{\phi}(x,y)
\end{split}
\end{align}
for $(x,y)\in\N^{2}$ and an arbitrary function $\phi:\N^{2}\to (0,1)$. Then the function $h$ in \eqref{eq2:various_functions_important} is harmonic for the associated transition operator $P$.
\end{lemma}

We are now able to provide the postponed proof of Prop.~\ref{prop:extinction IBCOS BUTS}.

\begin{proof}[of Prop.~\ref{prop:extinction IBCOS BUTS}]
By Lem.~\ref{lem:hom model} in combination with Lem.~\ref{lem1},
the assertion follows if we can verify that $(X_{n},Y_{n})_{n\ge 0}$ is standard in the sense of \eqref{eq:standard behavior}, i.e., $X_{n}\wedge Y_{n}\to\infty$ a.s.\ on $\{\tau=\infty\}$, where $\tau$ denotes the absorption time. For the \BUTS\ model, this follows in the same manner as for our hybrid model (see proof of Prop.~\ref{prop:standard}), but for the \IBCOS\ model, we need an extra argument. To this end, observe that $(X_{n},Y_{n})_{n\ge 0}$ may be obtained as the jump chain of a $2$-type Bellman-Harris process $(N_{t}^{\sfA},N_{t}^{\sfB})_{t\ge 0}$ with independent components, the latter being both single-type supercritical Bellman-Harris processes. Now, if neither $X_{n}$ nor $Y_{n}$ ever hits the axis, then the same must hold for $N_{t}^{\sfA}$ and $N_{t}^{\sfB}$, giving
$$ \lim_{n\to\infty}X_{n}\wedge Y_{n}\ =\ \lim_{t\to\infty}N_{t}^{\sfA}\wedge N_{t}^{\sfB}\ =\ \infty\quad\text{a.s.}$$
by the extinction-explosion principle.\qed
\end{proof}

\begin{acknowledgements}
Most of this work was done during mutual visits of the authors between 2015 and 2018 at their respective home institutions. Hospitality and excellent working conditions at these institutions are most gratefully acknowledged. The second author would also like to thank Viet Chi Tran for interesting discussions. We are indebted to two anonymous referees and an associate editor for very careful reading and many constructive remarks that helped to improve the original version of this article.
\end{acknowledgements}

\bibliographystyle{abbrvnat}      

\begin{thebibliography}{13}
\providecommand{\natexlab}[1]{#1}
\providecommand{\url}[1]{\texttt{#1}}
\expandafter\ifx\csname urlstyle\endcsname\relax
  \providecommand{\doi}[1]{doi: #1}\else
  \providecommand{\doi}{doi: \begingroup \urlstyle{rm}\Url}\fi

\bibitem[Alsmeyer and R{\"o}sler(1996)]{AlsRoe:96}
G.~Alsmeyer and U.~R{\"o}sler.
\newblock The bisexual {G}alton-{W}atson process with promiscuous mating:
  extinction probabilities in the supercritical case.
\newblock \emph{Ann. Appl. Probab.}, 6\penalty0 (3):\penalty0 922--939, 1996.

\bibitem[Alsmeyer and R{\"o}sler(2002)]{AlsRoe:02}
G.~Alsmeyer and U.~R{\"o}sler.
\newblock Asexual versus promiscuous bisexual {G}alton-{W}atson processes: the
  extinction probability ratio.
\newblock \emph{Ann. Appl. Probab.}, 12\penalty0 (1):\penalty0 125--142, 2002.

\bibitem[Baxter(1982)]{Baxter}
R.~J. Baxter.
\newblock \emph{Exactly solved models in statistical mechanics}.
\newblock Academic Press, Inc. [Harcourt Brace Jovanovich, Publishers], London,
  1982.

\bibitem[Billiard and Tran(2012)]{BilliardTran:12}
S.~Billiard and V.~C. Tran.
\newblock A general stochastic model for sporophytic self-incompatibility.
\newblock \emph{J. Math. Biol.}, 64\penalty0 (1-2):\penalty0 163--210, 2012.

\bibitem[Cont and de~Larrard(2013)]{ContLarrard:13}
R.~Cont and A.~de~Larrard.
\newblock Price dynamics in a {M}arkovian limit order market.
\newblock \emph{SIAM J. Financial Math.}, 4\penalty0 (1):\penalty0 1--25, 2013.

\bibitem[Daley et~al.(1986)Daley, Hull, and Taylor]{DaleyHullTaylor:86}
D.~J. Daley, D.~M. Hull, and J.~M. Taylor.
\newblock Bisexual {G}alton-{W}atson branching processes with superadditive
  mating functions.
\newblock \emph{J. Appl. Probab.}, 23\penalty0 (3):\penalty0 585--600, 1986.

\bibitem[Ernst and Grigorescu(2017)]{ErnstGrigorescu:17}
P.~A. Ernst and I.~Grigorescu.
\newblock Asymptotics for the time of ruin in the war of attrition.
\newblock \emph{Adv. in Appl. Probab.}, 49\penalty0 (2):\penalty0 388--410,
  2017.

\bibitem[Fayolle et~al.(1995)Fayolle, Malyshev, and Menshikov]{FayMalMen:95}
G.~Fayolle, V.~Malyshev, and M.~Menshikov.
\newblock \emph{Topics in the constructive theory of countable {M}arkov
  chains}.
\newblock Cambridge University Press, Cambridge, 1995.

\bibitem[Foddy(1984)]{Foddy:84}
M.~E. Foddy.
\newblock \emph{Analysis of {B}rownian motion with drift, confined to a
  quadrant by oblique reflection (diffusions, {R}iemann-{H}ilbert problem)}.
\newblock ProQuest LLC, Ann Arbor, MI, 1984.
\newblock Thesis (Ph.D.), Stanford University.

\bibitem[Jagers(1975)]{Jagers:75}
P.~Jagers.
\newblock \emph{Branching {P}rocesses {W}ith {B}iological {A}pplications}.
\newblock Wiley-Interscience [John Wiley \& Sons], London, 1975.
\newblock Wiley Series in Probability and Mathematical Statistics -- Applied
  Probability and Statistics.

\bibitem[Kurkova and Raschel(2011)]{KurkovaRaschel:11}
I.~Kurkova and K.~Raschel.
\newblock Random walks in {$(\Bbb Z_+)^2$} with non-zero drift absorbed at the
  axes.
\newblock \emph{Bull. Soc. Math. France}, 139\penalty0 (3):\penalty0 341--387,
  2011.

\bibitem[Lafitte-Godillon et~al.(2013)Lafitte-Godillon, Raschel, and
  Tran]{LafRaschTran:13}
P.~Lafitte-Godillon, K.~Raschel, and V.~C. Tran.
\newblock Extinction probabilities for a distylous plant population modeled by
  an inhomogeneous random walk on the positive quadrant.
\newblock \emph{SIAM J. Appl. Math.}, 73\penalty0 (2):\penalty0 700--722, 2013.

\bibitem[Menshikov et~al.(2017)Menshikov, Popov, and Wade]{MenPoWa:17}
M.~Menshikov, S.~Popov, and A.~Wade.
\newblock \emph{Non-homogeneous random walks}, volume 209 of \emph{Cambridge
  Tracts in Mathematics}.
\newblock Cambridge University Press, Cambridge, 2017.
\newblock Lyapunov function methods for near-critical stochastic systems.

\end{thebibliography}
\def\cprime{$'$}


\end{document}